\documentclass[leqno]{amsart}
\usepackage{amssymb}
\usepackage{amsmath, amsfonts, vmargin, enumerate}
\usepackage{dsfont}
\input xypic

\newtheorem{thm}{Theorem}[section]

\newtheorem{lem}[thm]{Lemma}
\newtheorem{defi}[thm]{Definition}
\newtheorem{remark}[thm]{Remark}
\newtheorem{example}[thm]{Example}
\newtheorem{pb}[thm]{Problem}

\newtheorem{theoA}{Theorem}
\newtheorem{coroA}[theoA]{Corollary}
\newtheorem{remA}[theoA]{Remark}

\newenvironment{rkA}{\begin{remA}\rm}{\end{remA}}

\newenvironment{rk}{\begin{remark}\rm}{\end{remark}}

\numberwithin{equation}{section}

\newcommand{\real}{{\mathbb R}}

\newcommand{\ent}{{\mathbb Z}}

\newcommand{\un}{{\mathds {1}}}

\newcommand{\A}{{\mathcal A}}

\newcommand{\I}{{\mathcal I}}
\newcommand{\K}{{\mathcal K}}
\renewcommand{\L}{{\mathcal L}}

\newcommand{\R}{{\mathcal R}}

\renewcommand{\S}{{\mathcal S}}

\newcommand{\V}{\mathcal{V}}

\renewcommand{\a}{\alpha}
\renewcommand{\b}{\beta}

\renewcommand{\d}{\delta}

\newcommand{\f}{\varphi}

\renewcommand{\l}{\lambda}
\renewcommand{\O}{\Omega}
\renewcommand{\o}{\omega}
\newcommand{\s}{\sigma}

\newcommand{\8}{\infty}
\newcommand{\el}{\ell}

\newcommand{\wt}{\widetilde}

\newcommand{\n}{\noindent}

\newcommand{\be}{\begin{eqnarray*}}
\newcommand{\ee}{\end{eqnarray*}}
\newcommand{\beq}{\begin{equation}}
\newcommand{\eeq}{\end{equation}}
\newcommand{\beqn}{\begin{equation*}}
\newcommand{\eeqn}{\end{equation*}}
\newcommand{\cqd}{\hfill$\Box$}

\begin{document}

\title[Weighted variation inequalities]{Weighted variation inequalities for differential operators and singular integrals in higher dimensions}

\thanks{{\it 2000 Mathematics Subject Classification:} Primary: 42B20, 42B25. Secondary: 46E30}
\thanks{{\it Key words:} Variation inequalities,  $A_p$ weights,  differential operators, singular integrals, vector-valued variation inequalities}

\author{Tao Ma}
\address{School of Mathematics and Statistics, Wuhan University, Wuhan 430072, China }
\email{tma.math@whu.edu.cn}

\author{Jos\'e Luis Torrea}
\address{Departamento de Matem\'aticas and ICMAT-CSIC-UAM-UCM-UC3M, Facultad de Ciencias, Universidad
Aut\'onoma de Madrid, 28049 Madrid, Spain }
\email{joseluis.torrea@uam.es}

\author{Quanhua Xu}
\address{School of Mathematics and Statistics, Wuhan University, Wuhan 430072, China and Laboratoire de Math{\'e}matiques, Universit{\'e} de Franche-Comt{\'e},
25030 Besan\c{c}on Cedex, France}
\email{qxu@univ-fcomte.fr}

\date{}
\maketitle

\begin{abstract}
 We prove weighted $q$-variation inequalities with $2<q<\infty$ for differential and singular integral operators in higher dimensions. The vector-valued extensions of these inequalities are also given.
 \end{abstract}

\bigskip


\section{Introduction and results}


We pursue our investigation of weighted variation inequalities for differential operators and singular integrals. The one dimensional case has been studied in our previous article \cite{MTX}. In the present one we consider the higher dimensional case. We show that most results of \cite{MTX} extend to higher dimensions. However, the arguments in $\real^d$ with $d\ge2$ are more complicated than those in the case of $d=1$. This is particularly true for the weighted weak type $(1, 1)$ inequalities. Their proofs require a very careful geometrical analysis of the kernels in consideration. On the other hand, one-sided  weighted variation inequalities for one-sided differential operators were obtained in \cite{MTX}. However, at the time of this writing, it is not clear for us how to show their higher dimensional extensions.

\medskip

Variation inequalities have been the subject of numerous recent research papers  in probability, ergodic theory and harmonic analysis.  The first variation inequality was proved by L\'epingle \cite{Le} for martingales (see also \cite{PX} for a different approach and related results). Bourgain \cite{Bour} proved the variation inequality  for the ergodic averages of a dynamic system. Bourgain's work has been considerably improved by subsequent works and largely extended to many other operators in ergodic theory (see, for instance, \cite{jrw-high, kaufman, LMX}) and harmonic analysis  (cf. e.g., \cite{CJRW2, DMC, GT, jsw,  MTo, OSTTW}).

\medskip

To state our results we need to recall some definitions. Let $1\le q<\8$ and $a=\{a_t\}_{t>0}$ be a family of complex numbers. The $q$-variation of $a$ is defined to be
 $$
 \|a\|_{v_q} =\sup\big(\sum^\infty_{j=0}|a_{t_j}-a_{t_{j+1}}|^q\big)^{1/q},
 $$
where the supremum runs over all increasing sequences $\{t_j\}$ of positive numbers. Let $v_q$ denote the space of all functions on $(0,\,\8)$ with finite $q$-variation. This is a Banach space modulo constant functions. The norm (or more precisely, seminorm)  $\|\,\|_{v_q}$ will be also denoted by $\V_q$: $\V_q(a)=\|a\|_{v_q}$.

Our first result concerns singular integral operators. Let $K$ be a kernel on $\real^d\times\real^d\setminus\{(x, x):x\in\real^d\}$. We will suppose that $K$ satisfies the following regularity conditions. There exist two constants $\d>0$ and $C>0$ such that
 \begin{itemize}
 \item[$({\rm K}_0$):] $\displaystyle |K(x, y)|\le \frac{C}{|x-y|^d}$ for $x\neq y$;
 \item[$({\rm K}_1$):] $\displaystyle  |K(x, y)-K(z,y)|\le \frac{C|x-z|^\d}{|x-y|^{d+\d}}$ for $|x-y|>2|x-z|$;
 \item[$({\rm K}_2$):] $\displaystyle |K(y, x)-K(y,z)| \le \frac{C|x-z|^\d}{|x-y|^{d+\d}}$ for $|x-y|>2|x-z|$.
 \end{itemize}
 By a slight abuse of notation, we will also use $K$ to denote the associated singular integral operator:
 $$K(f)(x)=\int_{\real^d}K(x, y)f(y)dy,\quad x\in\real^d.$$
For any $t>0$ let $K_t$ be  the truncated operator:
 $$K_t(f)(x)=\int_{|x-y|>t}K(x, y)f(y)dy.$$
Let $\K(f)(x)=\{K_t(f)(x)\}_{t>0}$. Thus $\K$ is an operator mapping functions on $\real^d$ to families of functions on $\real^d$. We will consider the $q$-variation of $\K f$ (relative to the variable $t$):
 $$\V_q\K (f)(x)=\|\K (f)(x)\|_{v_q}.$$
Thus the operator  $\V_q\K$ sends functions on $\real^d$ to nonnegative functions on $\real^d$.

\medskip

We next recall the definition of $A_p$ weights. Let $w$ be a positive  function on  $\real^d$.
\begin{enumerate}[$\bullet$]
 \item  $w\in A_p$ (with $1<p<\8$) if
 $$
 \sup_Q\frac1{|Q|}\int_{Q} w(x)dx \,\big(\frac1{|Q|}\int_{Q} w(x)^{-\frac{1}{p-1}}dx\big)^{p-1} <\8,$$
where the supremum runs over all cubes  in $\real^d$; All cubes in this paper are assumed to be open and with sides parallel to the axes.
 \item   $w\in A_1$ if
 $M(w)\le Cw$
for some constant $C$.
 \end{enumerate}
Here $M(f)$ denotes the usual Hardy-Littlewood maximal function of a locally integrable function:
 $$M(f)(x)=\sup_{x\in Q}\frac1{|Q|}\int_{Q} |f(y)|dy.$$
Muckenhoupt's celebrated characterization of $A_p$ weights  asserts that $w\in A_p$ if and only if the operator $M$ is bounded on $L^p(\real^d, w)$ for $1<p<\8$, and $w\in A_1$ if and only if  $M$ maps $L^1(\real^d, w)$ to $L^{1, \8}(\real^d, w)$. We refer to \cite{gar-rubio} for more information.

As usual, $f^{\sharp}$ denotes the sharp maximal function of $f$:
 $$f^{\sharp}(x)=\sup_{x\in Q}\frac1{|Q|}\Big|\int_{Q} f(x)-\frac1{|Q|}\int_{Q} f(y)dy\Big|dx.$$
The space $BMO(\real^d)$ consists of all $f$ such that $f^{\sharp}\in L^\8(\real^d)$ equipped with $\|f\|_*=\|f^{\sharp}\|_\8$.

\medskip

We are now ready to state the first result of the paper. The symbol $A\lesssim B$ will mean an inequality up to a constant that may depend on the indices $p, q, d$, the constants $C, \d$ in $({\rm K}_0)$-$({\rm K}_2$), the weights $w$, etc. but never on the functions $f$ on $\real^d$ or the points $x\in\real^d$ in consideration.

\begin{theoA}\label{vq CZ}
 Let $K$ be a  kernel on $\real^d$ satisfying $({\rm K}_0)$-$({\rm K}_2)$, and let $2<q<\8$. Assume that the operator $\V_q\K$ is of type $(p_0,p_0)$ for some $1<p_0<\8$:
  $$\int_{\real^d}\big(\V_q\K (f)(x)\big)^{p_0}dx\lesssim \int_{\real^d}|f(x)|^{p_0}dx,\quad\forall\; f\in L^{p_0}(\real^d).$$
Then
 \begin{enumerate}[\rm(i)]
 \item for $w\in A_1$
  $$w\big(\big\{x\in\real^d:\V_q\K (f)(x)>\l\big\}\big)\lesssim\frac1\l \int_{\real^d}|f(x)|w(x)dx,\quad\forall\; f\in L^1(\real^d, w),\;\forall\; \l>0;$$
 \item for $1< p<\8$ and $w\in A_p$
  $$\int_{\real^d}\big(\V_q\K (f)(x)\big)^pw(x)dx\lesssim \int_{\real^d}|f(x)|^pw(x)dx,\quad\forall\; f\in L^p(\real^d, w).$$
  \item for any weight $w$ such that $w^{-1}\in A_1$
  $$\big\|\big(\V_q\K (f)\big)^{\sharp}w\big\|_\8\lesssim \big\|fw\big\|_\8, \quad\forall\; f\in L^\8_c(\real^d),$$
 where $L^\8_c(\real^d)$ denotes the space of bounded measurable functions with compact support. In particular, $\V_q\K$ is bounded from $L^\8_c(\real^d)$ to $BMO(\real^d)$.
   \end{enumerate}
\end{theoA}

Note that a similar result was proved in \cite{HLP} independently and almost at the same time; however, the result of \cite{HLP}  concerns only smooth truncations of singular integrals. We would emphasize that the above theorem is new even in the unweighted case. With regard to this, compare it with \cite[Theorem~B]{CJRW2}.   The main interest of the weighted $L^\8$-BMO boundedness in part (iii) lies in the fact that it implies, by extrapolation, the type $(p, p)$ estimate in (ii) (see \cite{HMS}). On the other hand, (i) and extrapolation yield (ii) too.

\medskip

The proof of the above theorem can be adapted to the situation of differential operators. For  $t>0$ let $B_t$ denote the open ball in $\real^d$ of center  at the origin and radius $t$.   Given a locally integrable function $f$ on $\real^d$ define
  $$A_t (f)(x)= \frac{1}{|B_t|} \int_{B_t}f(x+y)~dy= \frac{1}{|B_t|} \int_{\real^d}f(y)\un_{B_t}(y-x)~dy, \quad x\in\real^d.$$
These are the central differential operators on $\real^d$. The term ``differential operator" refers here to Lebegue's classical differential theorem.  Let $\A(f)(x) = \{ A_t (f)(x)\}_{t>0}$. We then consider the $q$-variation of the family $\A(f)(x)$: $\V_q\A(f)(x)=\|\A(f)(x)\|_{v_q}$.  Jones {\it et al} proved in \cite{jrw-high} that the operator $\V_q\A$ is bounded on $L^p(\real^d)$ for $1<p\le 2$ and from $L^1(\real^d)$ into $L^{1,\8}(\real^d)$. The following theorem extends their result not only to all $p>2$ but also to the weighted case.

 \begin{theoA}\label{vq D}
 Let $2<q<\8$. Then
 \begin{enumerate}[\rm(i)]
 \item for $w\in A_1$
  $$w\big(\big\{x\in\real^d:\V_q\A (f)(x)>\l\big\}\big)\lesssim\frac1\l \int_{\real^d}|f(x)|w(x)dx,\quad\forall\; f\in L^1(\real^d, w),\;\forall\; \l>0;$$
 \item for $1< p<\8$ and $w\in A_p$
  $$\int_{\real^d}\big(\V_q\A (f)(x)\big)^pw(x)dx\lesssim \int_{\real^d}|f(x)|^pw(x)dx,\quad\forall\; f\in L^p(\real^d, w).$$
   \item for any weight $w$ such that $w^{-1}\in A_1$
  $$\big\|\big(\V_q\A (f)\big)^{\sharp}w\big\|_\8\lesssim \big\|fw\big\|_\8, \quad\forall\; f\in L^\8_c(\real^d).$$
 In particular, $\V_q\A$ is bounded from $L^\8_c(\real^d)$ to $BMO(\real^d)$.
  \end{enumerate}
\end{theoA}

\begin{rkA}
 In the above theorem, the family $\{B_t\}_{t>0}$ of balls can be replaced by the family $\{Q_t\}_{t>0}$ of cubes, where $Q_t$ is the cube centered at the origin and having side length equal to $t$.
\end{rkA}

Fundamental examples to which Theorem~\ref{vq CZ} applies are the Riesz transforms. More generally, it also applies to singular integrals with homogeneous kernels.

\begin{coroA}
 Let $\O$ be a function on the unit sphere $S^{d-1}$ such that
  $$\O\in L_1(S^{d-1})\quad\textrm{and}\quad \int_{S^{d-1}}\O(\theta)d\s(\theta)=0,$$
where  $d\s$ denotes surface measure on $S^{d-1}$.  Assume in addition that $\O$ belongs to the H\"older class of order $\a$ for some $\a>0$: 
 $$\sup_{\theta_1, \theta_2\in\O}\frac{|\O(\theta_1)-\O(\theta_2)|}{|\theta_1-\theta_2|^\a}<\8.$$
  Let
  $$K(x, y)=\frac{\O((x-y)/|x-y|)}{|x-y|^d},\quad x, y\in\real^d,\; x\neq y.$$
 Then for $2<q<\8$ the operator $\V_q\K$ is bounded on $L^p(\real^d, w)$ for $1<p<\8$ and $w\in A_p$, and from $L^1(\real^d, w)$ to  $L^{1, \8}(\real^d, w)$ for  $w\in A_1$.
 \end{coroA}

For the kernel $K$ in this corollary, Campbell {\it et al}  proved in  \cite{CJRW2} that $\V_q\K$ is of type $(p, p)$ for $1<p<\8$ and weak type $(1, 1)$. So the corollary follows immediately from Theorem~\ref{vq CZ}. Note that the Riesz transforms $R_j$ are included in the family of singular integrals considered in the corollary. Thus we get weighted variation inequalities for Riesz transforms too. Such inequalities for Riesz transforms were already obtained in \cite{GT} but only for some special weights. More precisely, if $1<p<\8$ and $w(x)=|x|^\a$ with $-1<\a<p-1$, then $\V_q\R_j$ is bounded on $L^p(\real^d, w)$. However,   the result of \cite{GT} has the additional important feature that the relevant constant is dimension free.

\medskip

Let us give an application of Theorem~\ref{vq D} to approximate identities.

\begin{coroA}
 Let $\f: \real^d\to [0, +\infty)$ be a radial and radially decreasing integrable function. Let
$\f_t(x)=\frac{1}{t^d}\f(\frac{x}{t})$ and $\Phi(f)(x)=\{\f_t*f(x)\}_{t>0}$. Then for $2<q<\8$ the operator $\V_q\Phi$ is bounded on $L^p(\real^d, w)$ for $1<p<\8$ and $w\in A_p$, and from $L^1(\real^d, w)$ to  $L^{1, \8}(\real^d, w)$ for  $w\in A_1$.
\end{coroA}

\begin{proof}
 By approximation we can assume that $\f$ is of the form: $\f=\sum_k\a_k\un_{B_{r_k}}$ with $\a_k>0$ (the sum being finite). Then
  $$\f_t*f(x)=\sum_k\a_k|B_{r_k}| A_{r_kt}(f)(x);$$
 whence
  $$\V_q\Phi(f)(x)\le \|\f\|_1\V_q\A(f)(x).$$
 Then Theorem~\ref{vq D} immediately implies the corollary.
 \end{proof}

In particular, for $\f(x)=e^{-|x|^2}$ (resp. $\f(x)=(1+|x|^2)^{-d/2}$), the convolutions $\{\f_t*f\}_t$ give rise to the heat (resp. Poisson) semigroup relative to the Laplacian of $\real^d$, up to a multiple constant. For this two examples, the above corollary goes back to \cite{CMMTV}

\medskip

Both Theorems~\ref{vq CZ} and \ref{vq D} can be extended to the vector-valued case. The following result for the differential operators  improves Fefferman-Stein's celebrated vector-valued. Hong and Ma \cite{HM} extend it to the case where the space $\el_\rho$ is replaced by any UMD lattice.

\begin{theoA}\label{vq vector}
Let $q>2$ and $1<\rho<\8$.
\begin{enumerate}[\rm (i)]
  \item Let $K$ be a  kernel on $\real^d$ satisfying $({\rm K}_0)$-$({\rm K}_2)$ and such that the operator $\V_q\K$ is of type $(p_0,p_0)$ for some $1<p_0<\8$. Let  $1\le p<\infty$ and  $w \in A_p$. Then 
  $$\int_{\real^d}\Big(\sum_n\big(\V_q \K(f_n)(x)\big)^\rho\Big)^{p/\rho}w(x)dx
  \lesssim \int_{\real^d}\Big(\sum_n|f_n(x)|^\rho\Big)^{p/\rho}w(x)dx$$
  for all  finite sequences $\{f_n\}_{n\ge1}\subset L^p(\real^d, w)$ with $1<p<\8$, and
  $$w\big(\big\{x\in\real^d: \sum_n\big(\V_q \K(f_n)(x)\big)^\rho>\l^\rho\big\}\big)
  \lesssim \frac1\l\int_{\real^d}\Big(\sum_n|f_n(x)|^\rho\Big)^{1/\rho}w(x)dx$$
for all  finite sequences $\{f_n\}_{n\ge1}\subset L^1(\real^d, w)$ and any $\l>0$.
 \item A similar statement holds for the operator $\V_q\A$ in place of $\V_q\K$.
\end{enumerate}
 \end{theoA}

\begin{rkA}
 The first version of this paper was written almost at the same time as \cite{MTX} in the fall of 2012. All previous results were proved in that version except the weak type $(1, 1)$ inequality of  Theorem~\ref{vq vector}, which has prevented us from finalizing the paper (more precisely, the obstruction concerned the proof of \eqref{boundary2} below). It is only recently that Guixiang Hong pointed to us that an argument of \cite{kzk} could help lift this obstruction. Note that the main result of \cite{kzk} is precisely the part of Theorem~\ref{vq vector} for the differential operators. Although adapted from the pattern set up in \cite{jrw-high}, its proof differs from ours. So the overlap between the two papers is not significant. 
 \end{rkA}


\medskip

The paper is organized as follows. In  sections~\ref{pf of weakA} and \ref{pf of pA} we prove  Theorem~\ref{vq CZ}. The proofs of the two parts (i) and (ii) depend one on another. More precisely, the proof of (ii) depends on the unweighted version of (i), and that of (i) on (ii). The proof of the weak type $(1, 1)$ inequality is quite technical and requires a careful geometrical analysis of the kernel. This proof is much more complicated than the corresponding one in the one dimensional case in \cite{MTX}. However, the proof of the type $(p, p)$ inequality does not differ too much from the one dimensional case. In  sections~\ref{pf of pB} and \ref{pf of weakB}, we present the proof of Theorem~\ref{vq D}. This proof is similar to that  of Theorem~\ref{vq CZ}. The last section is devoted to the proof of Theorem~\ref{vq vector}.


\section{Proof of Theorem~\ref{vq CZ}: weak type $(1,1)$}\label{pf of weakA}


In this section we prove the weak type $(1, 1)$ inequality of Theorem~\ref{vq CZ}. In fact, only the unweighted version of part (i), i.e., for $w\equiv1$ will be completely proved in this section. The full generality will be completed in the end of section~\ref{pf of pA}. This proof is long and technical. It is based on a careful geometrical analysis of the truncated kernels of the singular integral. Although we follow the general pattern set up in \cite{CJRW2}, our argument is subtler than that of \cite{CJRW2}. For instance, our treatment of the long variation is quite complicated, while the one of \cite{CJRW2} is rather straightforward.

\medskip

As usual, the  classical Calder\'on-Zygmund decomposition will play a crucial role in our proof. Let us state it below for later reference (cf. e.g., \cite[Theorem~II.1.12]{gar-rubio}).  Given  a cube $Q\subset\real^d$ let $\wt Q=5 \sqrt d\,Q$,  the cube with the same center as $Q$ but $5 \sqrt d$ times the side length.

\begin{lem}\label{CZD}
 Let  $f$ be a compactly supported integrable function on $\real^d$ and  $\lambda >0$.  Then there exists a finite  disjoint family $\{Q_i\}$ of dyadic cubes satisfying the following properties
 \begin{itemize}
 \item[\rm(i)] $|f| \le \lambda$ on $\O^c$, where $\displaystyle\O = \bigcup_iQ_i$;
 \item[\rm(ii)] $\displaystyle \lambda < \frac1{|Q_i|} \int_{Q_i} |f| \le 2^d\lambda$;
 \item[\rm(iiii)] $\displaystyle\O\subset\{x\in\real^d: M(f)(x) > \lambda\}$ and $\displaystyle\{x\in\real^d: M(f)(x) > 4^d\lambda\}\subset \wt\O$, where $\displaystyle\wt\O = \bigcup_i\wt Q_i$;
 \end{itemize}
Define
 \begin{align*}
 &g=f \textrm{ on } \O^c\quad\textrm{and}\quad  g= \frac1{|Q_i|} \int_{Q_i} f \, \textrm{ on } Q_i  \textrm{ for each }i,\\
 &b= \sum_i b_i, \textrm{ where } b_i = \big(f -\frac1{|Q_i|} \int_{Q_i} f  \big) \un_{Q_i} .
  \end{align*}
Then
 \begin{itemize}
 \item[\rm(iv)] $f=g+b$;
 \item[\rm(v)] $\|g\|_\8 \le 2^d\lambda$;
\item[\rm(vi)] for each $i$, $\displaystyle  \int_{\real^d} b_i= 0 $ and $\displaystyle \frac1{|Q_i|}\int_{\real^d}|b_i|\le 2^{d+1} \lambda$.
 \end{itemize}
 \end{lem}

We also require the following elementary fact which is to be compared with Cotlar's almost orthogonality lemma (see  \cite{CJRW1} for the case $r=2$).

\begin{lem}\label{AOL}
 Let $\{h_{k, j}\}_{k, j\in\ent}$ be a family of vectors in a Banach space $B$ and $\{\Delta_{j}\}_{j\in\ent}$  a family of nonnegative numbers with $\sum_{j\in\ent}\Delta_{j}<\8$.  Assume that $\|h_{k, j}\|\le \o(k-j)\Delta_j^{1/r}$ with $1<r<\8$ and $\o=\sum_{j\in\ent}\o(j)<\8$. Then
 $$\sum_{k\in\ent}\big\|\sum_{j\in\ent}h_{k,j}\big\|^r\le \o^r\sum_{j\in\ent}\Delta_j.$$
 \end{lem}

\begin{proof}
The proof is straightforward by the H\"older inequality. Indeed, letting $r'$ be the conjugate index of $r$, we have
  \begin{align*}
  \sum_{k\in\ent}\big\|\sum_{j\in\ent}h_{k,j}\big\|^r
  &\le \sum_{k\in\ent}\big(\sum_{j\in\ent}\o(k-j)\big)^{r/r'}\big(\sum_{j\in\ent}\o(k-j)^{-r/r'}\|h_{k,j}\big\|^r\big)\\
  &\le \o^{r/r'} \sum_{k\in\ent}\sum_{j\in\ent}\o(k-j)\Delta_j
  =\o^r\sum_{j\in\ent}\Delta_j.
   \end{align*}
 Thus we are done.
 \end{proof}

The following standard lemma will be used several times later on.

\begin{lem}\label{standard}
 Let $w$ be a locally integrable nonnegative function on $\real^d$, $x_0\in\real^d$ and $r>0, \a>0$. Then
 $$\int_{|x-x_0|>r}\frac{w(x)dx}{|x-x_0|^{d+\a}}\lesssim \frac{1}{r^\a}\,M(w)(y)\;\textrm{ for }\;  y\in x_0+B_r.$$
 Consequently, if $w\in A_1$, then
  $$\int_{|x-x_0|>r}\frac{w(x)dx}{|x-x_0|^{d+\a}}\lesssim \frac{1}{r^\a}\,w(y)\;\textrm{ for }\; a.e. \; y\in x_0+B_r.$$
 \end{lem}

 \begin{proof}
We have
  \begin{align*}
  \int_{|x-x_0|>r}\frac{w(x)dx}{|x-x_0|^{d+\a}}
  &\le  r^{-\a}\sum_{s=0}^\8 2^{-\a s}\big[\frac1{(2^{s}r)^d}\int_{2^sr<|x-x_0|\le 2^{s+1}r}w(x)dx\big]\\
  &\lesssim r^{-\a}\sum_{s=0}^\82^{-\a s}\big[\frac1{(2^{s+1}r)^d}\int_{|x-x_0|\le 2^{s+1}r}w(x)dx\big]\\
  & \lesssim r^{-\a} M(w)(y)\;\textrm{ for }\;  y\in x_0+B_r.
   \end{align*}
 The assertion is thus proved. 
  \end{proof}

Before  proceeding to the proof of the weak type $(1, 1)$ inequality of the operator $\V_q\K$, we need more notation. For an interval $I=(s,\,t]$ with   $0<s<t<\8$ we denote by $R_I$ the annulus $\{x\in\real^d: s<|x|\le t\}$ and let
 $$K_I(f)(x)=K_s(f)(x)-K_t(f)(x)=\int_{\real^d}K(x, y)\un_{R_I}(x-y)f(y)dy.$$

Let $f$ be a compactly supported integrable function on $\real^d$ and $\l>0$. We must control  the quantity $w(\{x: \V_{q}\K(f) (x) > \lambda \})$ by $\|f\|_{L^1(\real^d, w)}/\l$. By rescaling, we can assume that $\l=2$. Keeping the  notation in Lemma~\ref{CZD} (with $\l=2$),  we have
 $$
 w(\{x: \V_{q}\K(f) (x) > 2 \})\le w(\{x :\V_{q}\K(g)(x) >1 \})
 + w(\{x: \V_{q}\K(b) (x)> 1 \}).
 $$
We must control the two terms on the right hand side by $\int_{\real^d} |f(x)|w(x)dx$. It is here for the good part $g$ that we require that $w\equiv1$. Thus if $w\equiv1$, then by the $L^{p_0}$-boundedness of $\V_q\K$ we have
\begin{eqnarray}\label{good}
 \begin{array}{ccl}
  \begin{split} \begin{displaystyle}
 w(\{x :\V_{q}\K(g)(x) >1 \})\end{displaystyle}
 &=\begin{displaystyle}|\{x :\V_{q}\K(g)(x) >1 \}|\le\int_{\real^d} (\V_{q}\K(g))^{p_0}(x)dx\end{displaystyle}\\
 &\lesssim\begin{displaystyle}  \int_{\real^d} |g(x)|^{p_0}dx \lesssim \int_{\real^d} |f(x)|dx \end{displaystyle}.
 \end{split} \end{array}
 \end{eqnarray}
We will prove the above inequalities for a general $w\in A_1$ in the end of section~\ref{pf of pA}.

In the rest of this section, we again assume that $w$ is a general $A_1$ weight. We will treat the bad part and show
 $$ w(\{x: \V_{q}\K(b) (x)> 1 \})\lesssim \int_{\real^d} |f(x)|w(x)dx.$$
 A preliminary step toward this end is the following
 $$w(\{x: \V_{q}\K(b) (x)> 1 \}) \le w(\wt\O) +
 w(\{x :x \notin \wt\O, \V_{q}\K(b) (x) >1 \}).$$
By the doubling property of $w$ and the weak type $(1, 1)$ boundedness of $M$ for $A_1$ weights, we have
 $$
 w(\wt\O)\lesssim \sum_i w(Q_i) =w(\O)\le w(\{x\in\real: M(f)(x) >2\})
 \lesssim \int_{\real^d} |f(x)|w(x)dx.
 $$
So it remains to treat $w(\{x\in \wt\O^c: \V_{q}\K(b) (x) >1 \})$. For clarity we divide this technical part of the proof into several steps.

\medskip\n{\bf Step~1. Decomposition into interior and boundary sums}. Given  $x\notin \wt\O$ choose an increasing sequence $\{t_j\}$ such that
 $$\V_{q}\K(b) (x)\le 2\big(\sum_j\big|K_{(t_j,\,t_{j+1}]}(b)(x)\big|^q\big)^{1/q}.$$
Note that the sequence $\{t_j\}$ depends on $x$, as well as the sets $\I_1(I)$ and $\I_2(I)$ below. But for notational simplicity we will not mention $x$ explicitly in $\{t_j\}$ or  $\I_k(I)$ ($k=1,2$), which should not cause any ambiguity. 

Let $I_j=(t_j,\,t_{j+1}]$. The intervals $I_j$'s are pairwise disjoint.  Note that $K_{I_j}(b)(x)\neq0$ only if $x+R_{I_j}$ meets some cube $Q_i$. We consider two cases according to  $Q_i\subset x+R_{I_j} $ or $Q_i\cap(x+\partial R_{I_j}) \neq\emptyset$. For a given interval $I$ let
 $$\I_1(I)=\{i:  Q_i\subset x+R_I \}\quad\textrm{and}\quad
 \I_2(I)=\{i:Q_i\cap(x+\partial R_I) \neq\emptyset\}.$$
Then
 $$\big(\sum_j\big|K_{I_j}(b)(x)\big|^q\big)^{1/q}\le
  \big(\sum_{j}\big|\sum_{i\in\I_1(I_j)}K_{I_j}(b_i)(x)\big|^q\big)^{1/q}
   + \big(\sum_{j}\big|\sum_{i\in\I_2(I_j)} K_{I_j}(b_i)(x)\big|^q\big)^{1/q}.$$
According to \cite{CJRW2}, the first sum on the right-hand side is called the interior sum and the second the boundary sum. It then follows that
 \begin{align*}
  w(\{x \in \wt\O^c: \V_{q}\K(b) (x) > 1 \})
  &\le w(\{x \in \wt\O^c:  \sum_{j}\big|\sum_{i\in\I_1(I_j)}K_{I_j}(b_i)(x)\big|^q> \frac1{2^q} \})\\
   &+
   w(\{x \in \wt\O^c:  \sum_{j}\big|\sum_{i\in\I_2(I_j)}K_{I_j}(b_i)(x)\big|^q > \frac1{2^q} \}).
   \end{align*}
For the two last terms,   we will use  the $\el^1$ and $\el^2$ norms instead of the $\el^q$ norm, respectively:
 \begin{align*}
 \big(\sum_{j}\big|\sum_{i\in\I_1(I_j)}K_{I_j}(b_i)(x)\big|^q\big)^{1/q}&\le \sum_{j}\big|\sum_{i\in\I_1(I_j)}K_{I_j}(b_i)(x)\big|,\\
 \big(\sum_{j}\big|\sum_{i\in\I_2(I_j)}K_{I_j}(b_i)(x)\big|^q\big)^{1/q}&\le \big(\sum_{j}\big|\sum_{i\in\I_2(I_j)}K_{I_j}(b_i)(x)\big|^2\big)^{1/2}.
 \end{align*}
Thus we are led to proving the following two inequalities
 \begin{align*}
 &w(\{x \in \wt\O^c: \sum_{j}\big|\sum_{i\in\I_1(I_j)}K_{I_j}(b_i)(x)\big|> \frac1{2} \})\lesssim \int_{\real^d}|f(y)|w(y)dy,\\
 &w(\{x \in \wt\O^c:  \sum_{j}\big|\sum_{i\in\I_2(I_j)}K_{I_j}(b_i)(x)\big|^2> \frac1{4} \})\lesssim \int_{\real^d}|f(y)|w(y)dy.
  \end{align*}
 The first on the interior sum is easy and will be done in step~2. The second on the boundary sum is much harder and will be handled in steps~3-5.

\medskip\n{\bf Step~2. Estimate on the interior sum}.   Let $i\in\I_1(I_j)$, that is $Q_i\subset  x+R_{I_j}$. Since $b_i$  is of vanishing mean, we have
 $$K_{I_j}(b_i)(x)=\int_{\real^d}K(x, y)b_i(y)dy=\int_{\real^d}\big(K(x, y)-K(x, c_i)\big)b_i(y)dy,$$
where $c_i$ is the center of $Q_i$. Therefore, for $x\notin\wt\O$ by $({\rm K}_2$) we get
 \begin{align*}
 \sum_{j}\big|\sum_{i\in\I_1(I_j)}K_{I_j}(b_i)(x)\big|
 &\le\sum_i\int_{Q_i}\big|K(x, y)-K(x, c_i)\big|\,|b_i(y)|dy\\
 &\lesssim \sum_i\frac{l_i^\d}{|x-c_i|^{d+\d}}\int_{Q_i}|b_i(y)|dy\\
 &\lesssim \sum_i\frac{l_i^\d}{|x-c_i|^{d+\d}}\int_{Q_i}|f(y)|dy,
 \end{align*}
where $l_i$ denotes the side length of $Q_i$. Thus by Lemma~\ref{standard}, we deduce
 \begin{align*}
  w(\{x \in \wt\O^c : \sum_{j}\big|\sum_{i\in\I_1(I_j)}K_{I_j}(b_i)(x)\big| > \frac1{2} \})
  &\lesssim\int_{\wt\O^c} \sum_{j}\big|\sum_{i\in\I_1(I_j)}K_{I_j}(b_i)(x)\big|w(x)dx\\
 &\lesssim \int_{\wt\O^c}\sum_i\frac{l_i^\d}{|x-c_i|^{d+\d}}\int_{Q_i}|f(y)|dyw(x)dx\\
 &\lesssim\sum_i l_i^\d\int_{Q_i}|f(y)|\big[\int_{|x-c_i|>l_i}\frac{w(x)dx}{|x-c_i|^{d+\d}}\big]dy\\
 &\lesssim\sum_i \int_{Q_i}|f(y)|w(y)dy\\
 & \lesssim \int_{\real^d}|f(y)|w(y)dy.
  \end{align*}

\medskip\n{\bf Step~3. Separation of the boundary sum into long and short variations}.  We will handle this part by passing through long and short variations. For each $j$,  consider two cases:
 \begin{enumerate}[$\bullet$]
 \item Case 1: $I_j$ does not contain  any power of $2$;
 \item Case 2: $I_j$ contains powers of $2$.
 \end{enumerate}
In case~1, $I_j\subset (2^k,\, 2^{k+1}]$ for some $k\in\ent$. In case~2, letting $m_j=\min\{k: 2^k\in I_j\}$ and $n_j=\max\{k: 2^k\in I_j\}$, we divide $I_j$ into three subintervals: $(t_j,\, 2^{m_j}],\, (2^{m_j},\, 2^{n_j}]$ and $(2^{n_j},\, t_{j+1}]$ (noting that if $m_j=n_j$, the middle interval is empty). Then
 $$|K_{I_j}(b)(x)|^2\le 3\,\big(
 |K_{(t_j,\, 2^{m_j}]}(b)(x)|^2+|K_{(2^{m_j},\, 2^{n_j}]}(b)(x)|^2+|K_{(2^{n_j},\, t_{j+1}]}(b)(x)|^2\big).$$
We need only to keep the subintervals whose associated annuli translated by $x$ intersect some $Q_i$.  Accordingly, we introduce two  collections of intervals:
  \begin{enumerate}[$\bullet$]
 \item $\S$ consists of all intervals in case~1, and all $(t_j,\, 2^{m_j}]$, $(2^{n_j},\, t_{j+1}]$ in case~2 if their associated annuli translated by $x$ intersect some $Q_i$;
 \item $\L$ consists of all intervals  $(2^{m_j},\, 2^{n_j}]$ resulting in case~2 if their associated annuli translated by $x$ intersect some $Q_i$.
 \end{enumerate}
Note that $\S\cup\L$ is a disjoint family of intervals and for each $I\in\S\cup\L$ we have that $x+R_I$ contains no any $Q_i$ but meets some $Q_{i'}$.

The above discussion leads to
  \begin{align*}
 \big(\sum_{j}\big|\sum_{i\in\I_2(I_j)}K_{I_j}(b_i)(x)\big|^2\big)^{1/2}
 &\le \sqrt3 \big(\sum_{I\in\L}\big|\sum_{i\in\I_2(I)}K_{I}(b_i)(x)\big|^2\big)^{1/2}\\
 & +  \sqrt3 \big(\sum_{I\in\S}\big|\sum_{i\in\I_2(I)}K_{I}(b_i)(x)\big|^2\big)^{1/2} .
  \end{align*}
 The first sum on the right is the long variation and the second the short variation. Thus
 \begin{align*}
 w(\{x \in \wt\O^c: \sum_{j}\big|\sum_{i\in\I_2(I_j)}K_{I_j}(b_i)(x)\big|^2 >\frac1{4} \})
 &\le  w(\{x \in \wt\O^c: \sum_{I\in\L}\big|\sum_{i\in\I_2(I)}K_{I}(b_i)(x)\big|^2 >\frac1{48} \})\\
 &+
 w(\{x \in \wt\O^c:\sum_{I\in\S}\big|\sum_{i\in\I_2(I)}K_{I}(b_i)(x)\big|^2>\frac1{48} \}).
  \end{align*}
The last two measures on the long and short variations will be estimated in step~4 and step~5, respectively.

\medskip\n{\bf Step~4. Estimate of the long  variation}.    Let $I=(2^m,\, 2^n]\in\L$ (with $m<n$) and   $i\in\I_2(I)$. Then $Q_i$ intersects one (and only one) of the two spheres $x+\{y:|y|=2^m\}$ and $x+\{y:|y|=2^n\}$. If $Q_i$ intersects the former, then  $2^m>2\sqrt  d\,l_i$. Indeed, there exists $y\in Q_i$ such that $|y-x|=2^m$. Since  $x \notin \wt\O$, we have
 $$2^m=|x-y|\ge |x-c|-|y-c|>5\sqrt  d\,\frac{l_i}2 -\sqrt d\,\frac{l_i}2=2\sqrt  d\,l_i.$$
 Thus for any $z\in Q_i$,
  \begin{align*}
   |z-x|&\le|y-x| +|z-y|\le2^m+  \sqrt d\,l_i <2^{m+1},\\
   |z-x|&\ge|y-x| -|z-y|\ge2^m- \sqrt d\,l_i >2^{m-1}\,.
   \end{align*}
 Consequently, 
 $$Q_i\subset x+R_{(2^{m-1},\,2^{m+1}]}\subset \bigcup_{k=m-1}^{m+1} (x+R_{(2^k,\, 2^{k+1}]})\,.$$
We have  a similar assertion if $Q_i\cap(x+\{y:|y|=2^n\})\neq\emptyset$. Also note that if the latter case happens,  $Q_i$ cannot intersect $x+\{y:|y|=2^m\}$. Hence
 $$Q_i\subset \bigcup_{k=m-1}^{m+1} (x+R_{(2^k,\, 2^{k+1}]})\;\textrm{ or }\;
 Q_i\subset \bigcup_{k=n-1}^{n+1} (x+R_{(2^k,\, 2^{k+1}]}).$$
This shows that $Q_i$ does not meet $x+R_{(2^k,\, 2^{k+1}]}$ for any integer $k\in [m+1,\, n-1)$. Thus for such a $k$ we must have
 $$K_{(2^k,\, 2^{k+1}]}(b_i)(x)=0.$$
Therefore, we deduce that
 \begin{align*}
 \big|\sum_{i\in\I_2(I)}K_{I}(b_i)(x)\big|^2
 & \le 2 \sum_{k=m}^{n-1}\big|\sum_{i\in\I_2(I)}K_{(2^k,\, 2^{k+1}]}(b_i)(x)\big|^2\\
 &  \le 2 \sum_{k=m}^{n-1}\big(\sum_{i\in \I_2(D_k)}\big| K_{D_k}(b_i)(x)\big|\big)^2,
 \end{align*}
  where $D_k=(2^k,\, 2^{k+1}]$. Note that if $i\in \I_2(I)$, then the above term $K_{(2^k,\, 2^{k+1}]}(b_i)(x)\neq0$ only if $k=m$ or $n-1$.  We do this procedure for every $I\in\L$ and sum up all inequalities so obtained. Consequently, we have
  $$\sum_{I\in\L}\big|\sum_{i\in\I_2(I)}K_{I}(b_i)(x)\big|^2
   \le 2 \sum_{I\in\L}\sum_{k=m(I)}^{n(I)-1}\big(\sum_{i\in \I_2(D_k)}\big|K_{D_k}(b_i)(x)\big|\big)^2,$$
where $I=(2^{m(I)},\, 2^{n(I)}]$.  For any $k\in \ent$, let $\I_{2, k}=\I_2(D_k)$. Noting that the intervals in $\L$ are pairwise disjoint, we get
  \beq\label{long}
  \sum_{I\in\L}\big|\sum_{i\in\I_2(I)}K_{I}(b_i)(x)\big|^2
  \le 2\sum_{k\in\ent}\big(\sum_{i\in\I_{2,k}}\big|K_{D_k}(b_i)(x)\big|\big)^2.
  \eeq
Thus
 \begin{align*}
 w(\{x \in \wt\O^c: \sum_{I\in\L}\big|\sum_{i\in\I_2(I)}K_{I}(b_i)(x)\big|^2 >\frac1{48} \})
 \lesssim
 \int_{\wt\O^c}\sum_{k\in\ent}\big(\sum_{i\in\I_{2,k}}\big|K_{D_k}(b_i)(x)\big|\big)^2w(x)dx.
  \end{align*}
 
 Now let $\I_{j}=\{i: l_i=2^j\}$ (recalling that $l_i$ is the side length of $Q_i$)  and define
 $$h_{k,j}(x)=\sum_{i\in\I_{2,k}\cap\I_j}\big|K_{D_k}(b_i)(x)\big|.$$
It is important to note that if $h_{k,j}(x)\neq0$, then $k>j$. We have
 $$\sum_{k\in\ent}\big|\sum_{i\in\I_{2,k}}K_{D_k}(b_i)(x)\big|^2
 \le\sum_{k\in\ent}\big(\sum_jh_{k,j}(x)\big)^2.$$
Therefore, we are led to proving
  \begin{align*}
  \sum_{k\in\ent}\int_{\wt\O^c}\big(\sum_jh_{k,j}(x)\big)^2w(x)dx
  &\lesssim \sum_i\int_{Q_i}|f(y)|w(y)dy\\
  &=\sum_j\sum_{i\in\I_j}\int_{Q_i}|f(y)|w(y)dy\\
  &\;{\mathop=^{\rm def}}\;\sum_j\Delta_j.
  \end{align*}
Thus we are  in the situation of applying Lemma~\ref{AOL} (with $r=2$). By that lemma, it suffices to show
 $$\int_{\wt\O^c}h_{k,j}(x)^2w(x)dx\lesssim 2^{j-k}\Delta_j$$
for all $k>j$. Let $x\in\wt\O^c$.  By $({\rm K}_0)$ and Lemma~\ref{CZD}, we have
   \begin{align*}
   h_{k,j}(x)^2
  &\lesssim \frac1{2^{2dk}}\big(\sum_{i\in\I_{2,k}\cap\I_j}\int_{Q_i}|b_i(y)|dy\big)^2\\
  &\lesssim \big(\sum_{i\in\I_{2,k}\cap\I_j}|Q_i|\big)\cdot
  \big(\frac1{2^{2dk}}\sum_{i\in\I_{2,k}\cap\I_j}\int_{Q_i}|f(y)|dy\big)\\
  &\lesssim 2^{(d-1)k+j}\sum_{i\in\I_{2,k}\cap\I_j}\frac1{|x-c_i|^{2d}}\int_{Q_i}|f(y)|dy
  \end{align*}
(recalling that $c_i$ is the center of $Q_i$). Here we have used two facts. The first  is  that if $\I_{2,k}\neq\emptyset$ (then necessarily $k>j$), then
 \beq\label{control Q}
 \sum_{i\in\I_{2,k}\cap\I_j}|Q_i|\le\sum_{i\in \I_2(D_k)\cap\I_j}|Q_i|\lesssim 2^{(d-1)k+j}.
 \eeq
The second one is that for $i\in\I_{2,k}\cap\I_j\neq\emptyset$
 \beq\label{dist to center}
 \frac1{|x-c_i|}\approx \frac1{2^k}.
 \eeq
On the other hand, by the discussion at the beginning of the present step, we have
 $$|x-c_i|> 2^{k-1},\quad \forall\; i\in\I_{2,k}\cap\I_j.$$
Therefore, by Lemma~\ref{standard}
 \begin{align*}
 \int_{\wt\O^c}h_{k,j}(x)^2w(x)dx
 &\lesssim 2^{(d-1)k+j}\sum_{i\in\I_{2,k}\cap\I_j}
  \int_{Q_i}|f(y)|\Big[\int_{\wt\O^c}\frac{\un_{|x-c_i|>2^{k-1}}}{|x-c_i|^{2d}}w(x)dx\Big]dy\\
 &\lesssim 2^{(d-1)k+j}\sum_{i\in\I_{2,k}\cap\I_j}
  \int_{Q_i}|f(y)|\Big[\int_{|x-c_i|>2^{k-1}}\frac{w(x)dx}{|x-c_i|^{2d}}\Big]dy\\
 &\lesssim 2^{j-k}\sum_{i\in\I_{2,k}\cap\I_j} \int_{Q_i}|f(y)|w(y)dy\\
 &\le2^{j-k}\Delta_j.
 \end{align*}
This is the announced estimate on the weighted $L^2$ norm of $h_{k,j}$. We have thus finished the proof for the long variation part.

\medskip\n{\bf Step~5. Estimate of the short  variation}.   The argument for this part is similar to  that of the preceding step. Let $\S_k=\{I\in\S: I\subset(2^k,\, 2^{k+1}]\}$ and define
 $$g_{k, j}(x)^2=\sum_{I\in\S_k}\big|\sum_{i\in\I_2(I)\cap\I_j}K_{I}(b_i)(x)\big|^2.$$
The indices $k$ and $j$ again satisfy  $k>j$. Then
  $$
  \sum_{I\in\S}\big|\sum_{i\in\I_2(I)}K_{I}(b_i)(x)\big|^2
  = \sum_k\sum_{I\in\S_k}\big|\sum_{i\in\I_2(I)}K_{I}(b_i)(x)\big|^2
  \le \sum_k\big(\sum_jg_{k, j}(x)\big)^2.$$
  Thus
  \begin{align*}
  w(\{x \in \wt\O^c: \sum_{I\in\S}\big|\sum_{i\in\I_2(I)}K_{I}(b_i)(x)\big|^2>\frac1{48} \})
  &\lesssim \int_{\wt\O^c}\sum_{I\in\S}\big|\sum_{i\in\I_2(I)}K_{I}(b_i)(x)\big|^2w(x)dx\\
  &\lesssim \int_{\wt\O^c}\sum_k\big(\sum_jg_{k, j}(x)\big)^2w(x)dx.
  \end{align*}
Thanks again to Lemma~\ref{AOL}, we only need to show
  $$\int_{\wt\O^c}g_{k, j}(x)^2w(x)dx\lesssim 2^{j-k}\Delta_j,$$
where  $\Delta_j$ is the same as in the previous step. For $x\in\wt\O^c$, by  \eqref{control Q}, \eqref{dist to center} and the discussion following \eqref{dist to center}, we get
   \begin{align*}
   g_{k,j}(x)^2
  &\lesssim \frac1{2^{2dk}}\sum_{I\in\S_k}\big(\sum_{i\in\I_2(I)\cap\I_j}\int_{Q_i}\un_{x+R_I}(y)|b_i(y)|dy\big)^2\\
  &\lesssim \frac1{2^{2dk}}\sum_{I\in\S_k}\big(\sum_{i\in\I_2(I)\cap\I_j}|Q_i|\big)\cdot
  \big(\sum_{i\in\I_2(I)\cap\I_j}\int_{Q_i}\un_{x+R_I}(y)|b_i(y)|dy\big)\\
  &\lesssim 2^{(d-1)k+j}\sum_{I\in\S_k}\sum_{i\in\I_2(I)\cap\I_j}\frac1{2^{2dk}}\int_{Q_i}\un_{x+R_I}(y)|b_i(y)|dy\\
  &\lesssim 2^{(d-1)k+j}\sum_{I\in\S_k}\sum_{i\in\I_2(I)\cap\I_j}\frac{\un_{|x-c_i|>2^{k-1}}}{|x-c_i|^{2d}}\int_{Q_i}\un_{x+R_I}(y)|b_i(y)|dy\\
  &\lesssim 2^{(d-1)k+j}\sum_{i\in\I_j}\frac{\un_{|x-c_i|>2^{k-1}}}{|x-c_i|^{2d}}\sum_{I\in\S_k}\int_{Q_i}\un_{x+R_I}(y)|b_i(y)|dy.
  \end{align*}
Since the intervals in $\S$ are disjoint, so are the associated annuli. Thus
   \begin{align*}
   g_{k,j}(x)^2
  &\lesssim 2^{(d-1)k+j}\sum_{i\in\I_j}\frac{\un_{|x-c_i|>2^{k-1}}}{|x-c_i|^{2d}}\int_{Q_i}|b_i(y)|dy\\
  &\lesssim 2^{(d-1)k+j}\sum_{i\in\I_j}\frac{\un_{|x-c_i|>2^{k-1}}}{|x-c_i|^{2d}}\int_{Q_i}|f(y)|dy.
  \end{align*}
Integrating over $\wt\O^c$ and using Lemma~\ref{standard}, we then get the desired  weighted $L^2$ norm estimate of $g_{k,j}$. Hence, the estimate for the short variation is done.

Combining  the results proved in all preceding steps, we get
 $$w(\{x\in \wt\O^c:\V_{q}\K(b) (x) >1 \})\lesssim\int_{\real^d}|f(x)|w(x)dx,$$
so
 $$w(\{x : \V_{q}\K(b) (x) >1 \})\lesssim\int_{\real^d}|f(x)|w(x)dx.$$
This is the announced weighted estimate for the bad part $b$. Together with the unweighted estimate for the good part $g$ in the beginning of this proof, we finally prove the unweighted weak type $(1, 1)$ of $\V_q\K$.

The weighted weak type $(1, 1)$ of $\V_q\K$ will be proved in the end of section~\ref{pf of pA}.\cqd

\begin{rk}\label{close to 1}
 The $L^{p_0}$ boundedness in the assumption of Theorem~\ref{vq CZ} and the unweighted weak type $(1,1)$ of $\V_q\K$ just proved imply, via Marcinkiewicz's interpolation, that $\V_q\K$ is of type $(p, p)$ for any $1<p<p_0$.
 \end{rk}


\section{Proof of Theorem~\ref{vq CZ}: type $(p,p)$ and $L^\8$-BMO boundedness}\label{pf of pA}


This section is devoted to the proof of parts (ii) and (iii) of Theorem~\ref{vq CZ}. For a function $f$ on $\real^d$ and $r>1$ let $M_r(f)=M(|f|^r)^{1/r}$. Both (ii) and (iii)   will easily follow from the following inequality
 \beq\label{sharp K}
 \big(\V_q\K(f)\big)^{\sharp}\lesssim M_r(f)
 \eeq
for $1<r<\min(p_0, \, q)$.  Assuming \eqref{sharp K} for the moment, let us show (ii) and (iii). Let $w\in A_p$. It is well known that $w\in A_{p/r}$ too for some $r>1$. Then by \cite[Theorem~IV.2.20]{gar-rubio}, we have
 \begin{align*}
 \int_{\real^d}\big(\V_q\K(f)(x)\big)^pw(x)dx
 &\lesssim\int_{\real^d}\big((\V_q\K(f))^{\sharp}(x)\big)^pw(x)dx\\
 &\lesssim \int_{\real^d}\big(M_r(f)(x)\big)^pw(x)dx
 \lesssim \int_{\real^d}|f(x)|^pw(x)dx.
\end{align*}
Thus (ii) is proved. On the other hand,   assume that $w^{-1}\in A_1$. Choose $r>1$ such that $w^{-r}\in A_1$ too. Then for any cube $Q$  we have
 $$\frac1{|Q|}\int_Q|f(y)|^rdy
 \le\big\|fw\big\|_\8^r\frac1{|Q|}\int_Qw(y)^{-r}dy\lesssim \big\|fw\big\|_\8^r\,w(x)^{-r}\;\textrm{ for }\; a.e.\; x\in Q.$$
It follows that
 $$\big\| \big(\V_q\K(f)\big)^{\sharp}w\big\|_\8\lesssim \big\| M_r(f)w\big\|_\8\lesssim \big\| fw\big\|_\8.$$
This is the desired estimate in (iii)

\smallskip

Now we must prove  \eqref{sharp K}. To this end fix a compactly supported integrable function $f$ on $\real^d$ and a point $x_0\in\real^d$. We want to show
 $$ \left(\V_q\K(f)\right)^{\sharp}(x_0)\lesssim M_r(f) (x_0).$$
Recall that
 $$( \V_q\K(f))^{\sharp}(x_0)=
 \sup_{Q} \frac{1}{|Q|} \int_Q\Big| \V_q\K(f)(x)-\frac{1}{|Q|}\int_Q  \V_q\K(f)(y)dy\Big|dx,$$
where the supremum runs over all cubes $Q$ containing $x_0$. Fix such a cube $Q$ and let $c$ denote its center. Write $f=f_1+f_2$ with $f_1=f\un_{\wt Q}$ and $f_2=f\un_{\wt Q^c}$. Then
 \begin{align*}
  \frac{1}{|Q|}& \int_Q\Big| \V_q\K(f)(x)-\frac{1}{|Q|}\int_Q \V_q\K(f)(y)dy\Big|dx\\
 &\le  \frac{2}{|Q|} \int_Q\Big|\ \V_q\K(f)(x)- \V_q\K(f_2)(c)\Big|dx\\
 &\le \frac{2}{|Q|} \int_Q \V_q\K(f_1)(x)dx
 + \frac{2}{|Q|} \int_Q\|\K(f_2)(x)-\K(f_2)(c)\|_{v_q}dx\\
 &\;{\mathop=^{\rm def}}\;  D_1+D_2 .
 \end{align*}
We must show that $\max(D_1, \, D_2)\lesssim M_r(f)(x_0)$. This is easy for $D_1$. Indeed, by the H\"older inequality and the $L^r$-boundedness of $\mathcal V_q$ already observed in Remark~\ref{close to 1}, we have
  $$D_1\lesssim  \Big(\frac{1}{|Q|}\int_Q\big( \V_q\K(f_1)(x)\big)^rdx\Big)^{1/r}
   \lesssim \Big(\frac{1}{|Q|}\int_Q|f_1(x)|^rdx\Big)^{1/r}
   \lesssim M_r(f)(x_0).$$
To handle $D_2$  we will show
 \beq\label{D2}
 \big\|\K(f_2)(x)-\K(f_2)(c)\big\|_{v_r}\lesssim M_r(f)(x_0), \quad \forall\; x\in Q.
 \eeq
This will imply $D_2 \lesssim M_r(f)(x_0)$ since $\|\,\|_{v_q}\le \|\,\|_{v_r}$ for $r<q$. Fix an increasing sequence $\{t_j\}_{j\geq 0}$ of positive numbers and let $I_j=(t_j, \,t_{j+1}]$. Then
 \begin{align*}
  K_{I_j}(f_2)(x)-K_{I_j}(f_2)(c)
  &=\int_{\real^d}\big[K(x, y)\un_{R_{I_j}}(x-y)-K(c, y)\un_{R_{I_j}}(c-y)\big]f_2(y)dy\\
  &=\int_{\real^d}\big[K(x, y)-K(c, y)\big]\un_{R_{I_j}}(x-y)f_2(y)dy\\
  &\quad+\int_{\real^d}K(c, y)\big[\un_{R_{I_j}}(x-y)-\un_{R_{I_j}}(c-y)\big]f_2(y)dy\\
  &\;{\mathop=^{\rm def}}\; \a_j+\b_j.
  \end{align*}
The first term $\a_j$ is easy to estimate. Indeed, by $({\rm K}_1)$ and Lemma~\ref{standard}
 \begin{align*}
 \big(\sum_{j=0}^{\infty}|\a_j|^r\big)^{1/r}
 &\le \sum_{j=0}^{\infty}|\a_j|\\
 &\le  \sum_{j=0}^{\infty}\int_{\real^d}\big|K(x, y)-K(c, y)\big|\un_{R_{I_j}}(x-y)|f_2(y)|dy\\
 &\lesssim \sum_{j=0}^{\infty}\int_{\real^d}\frac{|x-c|^\d}{|y-c|^{d+\d}}\,\un_{R_{I_j}}(x-y)|f_2(y)|dy\\
 &\le |x-c|^\d\int_{|y-c|>l}\frac{1}{|y-c|^{d+\d}}|f(y)|dy\\
 &\lesssim M(f)(x_0)\le M_r(f)(x_0).
 \end{align*}
where $l$ denotes the side length of $Q$. Thus
 \beq\label{a}
 \big(\sum_{j=0}^{\infty}|a_j|^r\big)^{1/r}\lesssim  M_r(f)(x_0).
 \eeq

To deal with the second term $\b_j$ we introduce the following sets
 \beq\label{partition J}
 J_1=\big\{j\;:\; t_{j+1}-t_j\le |x-c|\big\}\quad\textrm{and}\quad J_2=\big\{j\;:\; t_{j+1}-t_j> |x-c|\big\}.
 \eeq
Then
 \beq\label{un J}
 \big|\un_{R_{I_j}}(x-y)-\un_{R_{I_j}}(c-y)\big| \le\un_{R_{I_j}}(x-y)+\un_{R_{I_j}}(c-y),j\in J_1
 \eeq
and
 \begin{align*}
  \big|\un_{R_{I_j}}(x-y)-\un_{R_{I_j}}(c-y)\big|
 &\le \un_{R_{(t_j,\, t_j+|x-c|}]}(x-y) +\un_{R_{(t_{j+1}, \,t_{j+1}+|x-c|]}}(x-y)\\
 &+ \un_{R_{(t_j,\, t_j+|x-c|]}}(c-y) +\un_{R_{(t_{j+1}, \,t_{j+1}+|x-c|]}}(c-y), j\in J_2.
 \end{align*}
We first consider the part on $J_1$. By $({\rm K}_0)$ and the H\"older inequality
 \begin{align*}
 \sum_{j\in J_1}|\b_j|^r
 &\lesssim \sum_{j\in J_1}\big( \int_{\real^d}\big|K(c, y)|\big(\un_{R_{I_j}}(x-y)+\un_{R_{I_j}}(c-y)\big)|f_2(y)|dy\big)^r\\
 &\lesssim \sum_{j\in J_1} (t_{j+1}^d-t_j^d)^{r-1}\int_{\real^d}\frac1{|y-c|^{dr}}\,\big(\un_{R_{I_j}}(x-y)+\un_{R_j}(c-y)\big)|f_2(y)|^rdy\\
 &\lesssim |x-c|^{r-1}\sum_{j\in J_1} t_{j+1}^{(d-1)(r-1)}\int_{\real^d}\frac1{|y-c|^{dr}}\,\big(\un_{R_{I_j}}(x-y)+\un_{R_{I_j}}(c-y)\big)|f_2(y)|^rdy.
  \end{align*}
For any $y\in\wt Q^c$ let $j(y)$ be the unique $j\in J_1$  such that $t_j<|x-y|\le t_{j+1}$ (if such a $j$ exists). Here we have used the pairwise disjointness of the annuli $R_{I_j}$'s.  Then
 $$t_{j(y)+1}\le t_{j(y)}+|x-c|\le 2|x-y|\lesssim |y-c|.$$
Thus  by Lemma~\ref{standard}
 \begin{align*}
 \sum_{j\in J_1} t_{j+1}^{(d-1)(r-1)}\int_{\real^d}\frac1{|y-c|^{dr}}\,\un_{R_{I_j}}(x-y)|f_2(y)|^rdy
 &=\int_{\real^d}\frac{ t_{j(y)+1}^{(d-1)(r-1)}}{|y-c|^{dr}}\,|f_2(y)|^{r}dy\\
 &\lesssim \int_{\real^d}\frac{1}{|y-c|^{d+r-1}}\,|f_2(y)|^rdy\\
 &\lesssim\int_{|y-c|>l}\frac1{|y-c|^{d+r-1}}|f(y)|^rdy\\
 &\lesssim l^{1-r}(M_r(f)(x_0))^r.
\end{align*}
This yields the desired estimate on the terms containing $\un_{R_{I_j}}(x-y)$. Taking $x=c$, we get the same estimate for  the terms containing $\un_{R_{I_j}}(c-y)$. Therefore,
 \beq\label{b1}
 \big(\sum_{j\in J_1}|\b_j|^r\big)^{1/r}\lesssim M_r(f)(x_0).
 \eeq
The part on $J_2$ is treated in a similar way. Indeed,
 \begin{align*}
 \sum_{j\in J_2}|\b_j|^r
 &\lesssim \sum_{j\in J_2} \big(\int_{\real^d}\big|K(c,y)|\big(\un_{R_{t_j, t_j+|x-c|}}(x-y) +\un_{R_{t_j, t_j+|x-c|}}(c-y) \big)|f_2(y)|dy\big)^r\\
 &\lesssim |x-c|^{r-1} \sum_{j\in J_2}(t_j+|x-c|)^{(d-1)(r-1)}\\
 &\quad \cdot\int_{|y-c|>l}\frac1{|y-c|^{dr}}
 \big(\un_{R_{t_j, t_j+|x-c|}}(x-y) +\un_{R_{t_j, t_j+|x-c|}}(c-y) \big)|f_2(y)|^rdy.
   \end{align*}
Since the family $\{x+R_{t_j, t_j+|x-c|}\}_{j\in J_2}$ is disjoint, for any $y$ there exists at most one $j\in J_2$ such that $t_j<|x-y|\le t_j+|x-c|$. Denote such a $j$ still by $j(y)$. Then we have $t_j+|x-c|\lesssim |y-c|$ as in the preceding case for $J_1$. Thus we conclude as before that
 \beq\label{b2}
 \big(\sum_{j\in J_2}|\b_j|^r\big)^{1/r}\lesssim M_r(f)(x_0).
 \eeq
 Combining \eqref{a},  \eqref{b1} and \eqref{b2}, we get
  $$\big(\sum_j\big|K_{I_j}(f_2)(x)-K_{I_j}(f_2)(c)\big|^r\big)^{1/r}\lesssim M_r(f)(x_0).$$
 Taking the supremum over all increasing sequences $\{t_j\}$ yields \eqref{D2}. We have thus proved  part (ii) of Theorem~\ref{vq CZ}. \cqd

\medskip\n{\bf End of the proof of part (i).}  Let us go back to the full generality of part (i). As already noted in section~\ref{pf of weakA}, the only missing point is the weighted estimate for the good part $g$ in \eqref{good}. The ingredient for this estimate is the weighted type $(p_0, p_0)$ of $\V_q\K$ with respect to any weighted $w\in A_{p_0}$. Now  part (ii)  makes this at our disposal. So for $w\in A_1\subset A_p$, by the properties of $g$ in Lemma~\ref{CZD} and the weak type $(1, 1)$ of the maximal operator $M$, we have
  \begin{align*}
 w(\{x :\V_{q}\K(g)(x) >1 \})
 &\le\int_{\real^d} (\V_{q}\K(g))^{p_0}(x)w(x)dx\lesssim  \int_{\real^d} |g(x)|^{p_0}w(x)dx\\
 &\lesssim\int_{\O^c} |g(x)|w(x)dx +w(\O) \lesssim \int_{\real^d} |f(x)|w(x)dx.
 \end{align*}
Thus \eqref{good} also holds in the weighted case, so we have proved part (i) too. Thus the proof of Theorem~\ref{vq CZ} is complete. \cqd

\medskip
Reexamining the proof of Theorem~\ref{vq CZ} we get the following

\begin{rk}
 Assume that $\V_q\K$ is bounded on $L^{p_0}(\real^d)$ for some $1<p_0<\8$.
 \begin{enumerate}[(i)]
 \item If the kernel $K$ satisfies $({\rm K}_0)$ and $({\rm K}_2)$, then $\V_q\K$ is bounded  on $L^{p}(\real^d, w)$ for $1<p<p_0$ and $w\in A_p$, and from $L^{1}(\real^d, w)$ to $L^{1, \8}(\real^d, w)$ for $w\in A_1$.
 \item If the kernel $K$ satisfies $({\rm K}_0)$ and $({\rm K}_1)$, then $\V_q\K$ is bounded  $L^{p}(\real^d, w)$ for $p_0<p<\8$ and $w\in A_{p/p_0}$
 \end{enumerate}
\end{rk}


\section{Proof of Theorem~\ref{vq D}: type $(p, p)$ and $L^\8$-BMO boundedness}\label{pf of pB}


In this section we prove parts (ii) and (iii) of Theorem~\ref{vq D}. As in section~\ref{pf of pA}, it suffices to show
 \beq\label{DMr}
  \left(\V_q\A(f)\right)^{\sharp}\lesssim M_r(f)
  \eeq
for $r>1$ close to $1$. Fix a compactly supported  integrable function $f$ on $\real^d$ and a point $x_0\in\real^d$. Let $Q$ be a cube containing $x_0$. Let $c$ and $l$ denote respectively the center and side length of $Q$. We then  decompose $f$ as $f=f_1+f_2$ with $f_1=f\un_{\wt Q}$ and $f_2=f\un_{\wt Q^c}$. The part on $f_1$ is treated by using the boundedness of $\V_q\A$ on $L^2(\real^d)$ from \cite{jrw-high}. For the part on $f_2$ we will prove the following pointwise estimate
 \beq\label{Ineq:sharp2}
 \big\|\A (f_2)(x)-\A (f_2)(c)\big\|_{v_r}\lesssim M_r(f)(x_0)\quad \forall\;x\in Q,
 \eeq
which, in turn, will imply \eqref{DMr}.

Fix $x\in Q$. Note that $A_t(f_2)(x)=0$ for $t\le l$ (in fact, for  $t\le 2\sqrt d\,l$). So only the values of  $t$ greater than $l$ are relevant. Given an interval $I=(s,\,t]$, put $A_{I}(f)=A_{t}(f)-A_{s}(f)$ as before for singular integral operators. Let $\{t_j\}_j$ be an increasing sequence with $t_0>l$. Set $I_j=(t_j,\,t_{j+1}]$. Then
 \beq\label{dec}
  A_{I_j}(f_2)(x)-A_{I_j}(f_2)(c)=\xi_j+\eta_j,
 \eeq
where
 \begin{align*}
 \xi_j&=\frac{1}{|B_{t_{j+1}}|}\int_{\real^d} f_2(y)\big[\un_{R_{I_j}}(y-x)-\un_{R_{I_j}}(y-c)\big]dy,\\
 \eta_j&=\Big(\frac1{|B_{t_{j+1}}|}-\frac1{|B_{t_{j}}|}\Big)\int_{\real^d} f_2(y)\big[\un_{B_{t_j}}(y-x)-\un_{B_{t_j}}(y-c)\big]dy.
 \end{align*}
To handle  the $\xi_j$'s. we use the partition given by \eqref{partition J}. Then we have to show
 \beq\label{Ineq:xir}
 \big(\sum_{j\in J_1}|\xi_j|^r\big)^{1/r}\lesssim M_r(f)(x_0)\;\textrm{ and }\;
 \big(\sum_{j\in J_2}|\xi_j|^r\big)^{1/r}\lesssim M_r(f)(x_0).
 \eeq
Let us deal with only the part on $J_1$. Using \eqref{un J}, we need only to consider the terms on $x$ since those on $c$ are their special cases when $x=c$. Let $j(y)$ be the unique $j$ satisfying $t_{j(y)}\leq |y-x|<t_{j(y)+1}$ for a given $y\in \wt Q^c$. Then
  \begin{align*}
 \sum_{j\in J_1}\frac{1}{|B_{t_{j+1}}|^r}
 &\Big|\int_{\real^d} f_2(y)\un_{R_{I_j}}(y-x)dy\Big|^r\\
 &\lesssim\sum_{j\in J_1}\frac{1}{| t_{j+1}|^{dr}}(t_{j+1}^d-t_{j}^d)^{ r-1}\int_{\real^d} |f_2(y)|^r\un_{R_{I_j}}(y-x)dy\\
 &\lesssim|x-c|^{r-1}\sum_{j\in J_1} \frac{1}{ t_{j+1}^{d+r-1}}\int_{\real^d} |f_2(y)|^r\un_{R_{I_j}}(y-x)dy\\
 &\lesssim|x-c|^{r-1}\int_{\real^d} |f_2(y)|^r\frac{1}{t_{j(y)+1}^{d+r-1}}dy\\
 &\lesssim|x-c|^{r-1}\int_{|y-c|>l} |f(y)|^r\frac{1}{|y-c|^{d+r-1}}dy\\
 &\lesssim (M_r(f)(x_0))^r.
\end{align*}
This finishes the estimate on the $\xi_j$'s.

Now we turn to  the $\eta_j$'s. Observe that
 $$\big|\un_{B_{t_j}}(y-x)-\un_{B_{t_j}}(y-c)\big| \leq \un_{R_{(t_j, \,t_j+|x-c|]}}(y-x)+\un_{R_{(t_j, \,t_j+|x-c|]}}(y-c)$$
Recall that $t_j>l\ge |x-c|/\sqrt{d}$ for every $x\in Q$ by the assumption on the sequence $\{t_j\}$ at the beginning of the proof. Then by the H\"older inequality
 \begin{align*}
 \Big|\int_{\real^d} &f_2(y)\big[\un_{B_{t_j}}(y-x)-\un_{B_{t_j}}(y-x_0)\big]dy\Big|^r\\
 &\lesssim \big((t_{j}+|x-c|)^d-t_j^d\big)^{r-1}
 \int_{\real^d} |f_2(y)|^r\big(\un_{R_{(t_j, \,t_j+|x-c|]}}(y-x)+\un_{R_{(t_j,\, t_j+|x-c|]}}(y-c)\big)dy.
  \end{align*}
Hence as before, we deduce
 \begin{align*}
 \sum_j|\eta_j|^r
 &\lesssim\sum_j\frac{ |x-c|^{r-1}}{t_j^{d+r-1}}
  \int_{\real^d} |f_2(y)|^r\big(\un_{R_{(t_j, \,t_j+|x-c|]}}(y-x)+\un_{R_{(t_j,\, t_j+|x-c|]}}(y-c)\big)dy\\
 &\lesssim|x-c|^{r-1}\int_{|y-c|>l} |f(y)|^r\frac{1}{|y-c|^{d+r-1}}dy\\
 &\lesssim (M_r(f)(x_0))^r.
\end{align*}
Therefore, we have treated both $\xi_j$ and $\eta_j$ in \eqref{dec}.
Combining this inequality with \eqref{dec} and \eqref{Ineq:xir}, we finally get
 $$\big(\sum_j|A_{I_j}(f_2)(x)-A_{I_j}(f_2)(c)|^r\big)^{1/r} \lesssim M_r(f)(x_0);$$
whence \eqref{Ineq:sharp2}. Thus parts (ii) and (iii) of Theorem~\ref{vq D} is proved. \cqd


\section{Proof of Theorem~\ref{vq D}: weak type $(1,1)$}\label{pf of weakB}


This section is devoted to the proof of the weak type $(1, 1)$ inequality of Theorem~\ref{vq D}. This proof is similar to the one presented in section~\ref{pf of weakA}.  So we will only give the main lines and indicate the differences.

Let $w\in A_1$ and  $f$ be a compactly supported integrable function on $\real^d$. We want to show
 $$w(\{x: \V_{q}\A(f) (x) > \lambda \})\lesssim\frac1\l\int_{\real^d}f(x)w(x)dx,\quad\forall\;\l>0.$$
Let $f=g+b$ be the Calder\'on-Zygmund decomposition of $f$ given by Lemma~\ref{CZD} with the associated dyadic cubes $\{Q_i\}$. We keep all notation introduced in section~\ref{pf of weakA} relative to this decomposition.

As in the end of section~\ref{pf of pA}, the good part $w(\{x: \V_{q}\A(g) (x) > \l/2 \})$ is estimated by the boundedness of the operator $\V_q\A$ on $L^2(\real^d, w)$ proved in  section~\ref{pf of pB}. For the bad part  we need only to majorize the part of $w(\{x: \V_{q}\A(b) (x) > \l/2 \})$ outside of $\wt\O^c$. Thus our remaining task is to show the following inequality
 $$w(\{x\in\wt\O^c:\V_{q}\A(b) (x) > \frac\l2 \})\lesssim\frac1\l\int_{\real^d}f(x)w(x)dx.$$
Considering $4f/\l$ instead of $f$, we can assume that $\l=4$ in the rest of the section.

We start our majorization of $w(\{x\in\wt\O^c:\V_{q}\A(b) (x) >2 \})$ with an analysis of $A_I(b_i)(x)$ for an interval $I=(s,\, t]$.  Clearly, $A_I(b_i)(x)=0$ if $Q_i$ is outside of the ball $x+B_t$. On the other hand, since $b_i$ is of vanishing mean, $A_I(b_i)(x)=0$ if $Q_i$ is contained in the ball $x+B_s$ or in the annulus $x+R_I$. Thus $A_I(b_i)(x)\neq0$ only if $Q_i$ meets the boundary of $x+R_I$. This is a difference with the singular integrals: the interior sum disappears. So we denote $\I_2(I)$  simply by $\I(I)$:
 $$\I(I)=\{i: Q_i\cap(x+\partial R_I)\neq\emptyset\}.$$
$\I(I)$ depends on $x$ too. But for notational simplicity, we omit $x$ as an index in $\I(I)$ as well as in the sequence $\{t_j\}$ below.

Now for every $x\in \wt\O^c$ choose an increasing sequence $\{t_j\}$ such that
 $$\V_{q}\A(b) (x)\le 2\big(\sum_j\big|A_{(t_j,\,t_{j+1}]}(b)(x)\big|^q\big)^{1/q}.$$
Let $I_j=(t_j,\,t_{j+1}]$. By the above analysis, we have
 $$A_{I_j}(b)(x)=\sum_{i\in\I(I_j)}A_{I_j}(b_i)(x).$$
Then
 $$\V_{q}\A(b) (x)\le 2\big(\sum_j\big|\sum_{i\in\I(I_j)}A_{I_j}(b_i)(x)\big|^2\big)^{1/2}.$$
Thus we must show
 $$w(\{x\in\wt\O^c: \sum_j\big|\sum_{i\in\I(I_j)}A_{I_j}(b_i)(x)\big|^2> 1 \})\lesssim \int_{\real^d}f(x)w(x)dx.$$
Like in the case of singular integrals, we will do this via long and short variations.

Let $\L$ and $\S$ be the two collections of intervals associated to $\{t_j\}$ introduced in step~3 of section~\ref{pf of weakA}. Then
  \begin{align*}
 \big(\sum_{j}\big|\sum_{i\in\I(I_j)}A_{I_j}(b_i)(x)\big|^2\big)^{1/2}
 &\le \sqrt3 \big(\sum_{I\in\L}\big|\sum_{i\in\I(I)}A_{I}(b_i)(x)\big|^2\big)^{1/2}\\
 & +  \sqrt3 \big(\sum_{I\in\S}\big|\sum_{i\in\I(I)}A_{I}(b_i)(x)\big|^2\big)^{1/2} .
  \end{align*}
Hence
 \begin{align*}
 w(\{x\in\wt\O^c: \sum_{j}\big|\sum_{i\in\I(I_j)}A_{I_j}(b_i)(x)\big|^2 >1 \})
 &\le  w(\{x \in\wt\O^c: \sum_{I\in\L}\big|\sum_{i\in\I(I)}A_{I}(b_i)(x)\big|^2 >\frac1{12} \})\\
 &+
 w(\{x \in\wt\O^c: \sum_{I\in\S}\big|\sum_{i\in\I(I)}A_{I}(b_i)(x)\big|^2>\frac1{12} \}).
  \end{align*}

We first deal with the long  variation.    The geometrical analysis made at the beginning of step~4 in section~\ref{pf of weakA} remains valid now. Maintaining the notation there, we then have
  $$\sum_{I\in\L}\big|\sum_{i\in\I(I)}A_{I}(b_i)(x)\big|^2
  \le2 \sum_{k\in\ent}\big|\sum_{i\in\I_{2,k}}A_{D_k}(b_i)(x)\big|^2.$$
Recalling that $\I_{j}=\{i: l_i=2^j\}$, define again
 $$h_{k,j}(x)=\sum_{i\in\I_{2,k}\cap\I_j}\big|A_{D_k}(b_i)(x)\big|.$$
In order to apply Lemma~\ref{AOL}, we have to estimate $h_{k,j}(x)$. Note that the argument in step~4 of section~\ref{pf of weakA} for this estimate is purely geometrical  except one place where $({\rm K}_0)$ of the kernel $K$ is used. Now the corresponding differential operator kernel is
 $$\frac1{|B_{2^{k+1}}|}\,\un_{B_{2^{k+1}}}(x-y)-\frac1{|B_{2^{k}}|}\,\un_{B_{2^{k}}}(x-y)$$
and
 $$A_{D_k}(b_i)(x)=\frac1{|B_{2^{k+1}}|}\int_{\real^d}\un_{B_{2^{k+1}}}(x-y)b_i(y)dy
 -\frac1{|B_{2^{k}}|}\int_{\real^d}\un_{B_{2^{k}}}(x-y)b_i(y)dy.$$
Thus we still have
 $$|A_{D_k}(b_i)(x)|\lesssim\frac1{2^{dk}}\int_{Q_i}|b_i(y)|dy.$$
Then as in step~4 of section~\ref{pf of weakA}, we deduce
 $$\int_{\wt\O^c}h_{k,j}(x)^2w(x)dx\lesssim 2^{j-k}\Delta_j=2^{j-k}\sum_{i\in\I_j}\int_{Q_i}|f(y)|w(y)dy.$$
Therefore, by Lemma~\ref{AOL}, we get the desired estimate on the long variation:
 \begin{align*}
 &w(\{x \in \wt\O^c: \sum_{I\in\L}\big|\sum_{i\in\I(I)}A_{I}(b_i)(x)\big|^2 >\frac1{12} \})\\
 &\quad\lesssim  \int_{\wt\O^c}\sum_{k\in\ent}\big(\sum_jh_{k,j}(x)\big)^2w(x)dx\\
 &\quad\lesssim \int_{\real^d}|f(y)|w(y)dy.
  \end{align*}

We turn to the short  variation.    Define
 $$g_{k, j}(x)=\big(\sum_{I\in\S_k}\big|\sum_{i\in\I(I)\cap\I_j}A_{I}(b_i)(x)\big|^2\big)^{1/2}$$
(recalling that $\S_k=\{I\in\S: I\subset(2^k,\, 2^{k+1}]\}$). As in step~5 of section~\ref{pf of weakA} and the above argument for the long variation, we need only to show the following inequality
 $$ g_{k,j}(x)^2
 \lesssim 2^{(d-1)k+j}\sum_{i\in\I_j}\frac{\un_{|x-c_i|>2^{k-1}}}{|x-c_i|^{2d}}\int_{Q_i}|f(y)|dy,\quad\forall\; x\in\wt\O^c.$$
Let $I=(l(I),\, r(I)]\in\S_k$. Then
 \begin{align*}
 A_I(b_i)(x)
 &=\frac1{|B_{r(I)}|}\int_{\real^d}\un_{R_{I}}(x-y)b_i(y)dy
 +\big(\frac1{|B_{r(I)}|} -\frac1{|B_{l(I)}|}\big)\int_{\real^d}\un_{B_{l(I)}}(x-y)b_i(y)dy\\
 &\;{\mathop =^{\rm def}}\;A_I^{(1)}(b_i)(x)+A_I^{(2)}(b_i)(x).
 \end{align*}
Let
 $$g^{(\a)}_{k,j}(x)=\big(\sum_{I\in\S_k}\big|\sum_{i\in\I(I)\cap\I_j}A^{(\a)}_{I}(b_i)(x)\big|^2\big)^{1/2},\quad \a=1, 2.$$
Then $g_{k,j}(x)\le g^{(1)}_{k,j}(x)+g^{(2)}_{k,j}(x)$. Thus we are reduced to estimating $g^{(1)}_{k,j}(x)$ and $g^{(2)}_{k,j}(x)$ separately. The estimate of $g^{(1)}_{k,j}(x)$ is done in the same way as before for the singular integrals, since the kernel of $A_I^{(1)}$ behaves like a singular kernel as far as such an estimate is concerned.

Compared with the situation of section~\ref{pf of weakA}, the second term is  new. We have
 $$|A_I^{(2)}(b_i)(x)|\lesssim \frac{r(I)-l(I)}{2^{(d+1)k}}\int_{Q_i}|b_i(y)|dy.$$
Thus by \eqref{control Q},
   \begin{align*}
   g^{(2)}_{k,j}(x)^2
  &\lesssim \frac1{2^{2(d+1)k}}\sum_{I\in\S_k}\big(r(I)-l(I)\big)^2\big(\sum_{i\in\I_2(I)\cap\I_j}\int_{Q_i}|b_i(y)|dy\big)^2\\
  &\lesssim \frac1{2^{(2d+1)k}}\sum_{I\in\S_k}\big(r(I)-l(I)\big)\big(\sum_{i\in\I_2(I)\cap\I_j}|Q_i|\big)\cdot
  \big(\sum_{i\in\I_2(I)\cap\I_j}\int_{Q_i}|b_i(y)|dy\big)\\
  &\lesssim \frac{2^{(d-1)k+j}}{2^{(2d+1)k}}\sum_{I\in\S_k}\big(r(I)-l(I)\big)\sum_{i\in\I_j}\int_{Q_i}|b_i(y)|dy\\
  &\lesssim \frac{2^{(d-1)k+j}}{2^{2dk}}\sum_{i\in\I_j}\int_{Q_i}|b_i(y)|dy\\
  &\lesssim 2^{(d-1)k+j}\sum_{i\in\I_j}\frac{\un_{|x-c_i|>2^{k-1}}}{|x-c_i|^{2d}}\int_{Q_i}|f(y)|dy.
  \end{align*}
Here we have used the fact the intervals $I$ in $\S_k$ are disjoint subintervals of $(2^k,\, 2^{k+1}]$. So
 $$\sum_{I\in\S_k}\big(r(I)-l(I)\big)\le 2^k.$$
Therefore, we obtain the announced estimate of $g_{k,j}(x)$. Along with Lemma~\ref{AOL}, this yields
 \begin{align*}
 &w(\{x \in \wt\O^c: \sum_{I\in\S}\big|\sum_{i\in\I(I)}A_{I}(b_i)(x)\big|^2 >\frac1{12} \})\\
 &\quad\lesssim  \int_{\wt\O^c}\sum_{k}\big(\sum_jg_{k,j}(x)\big)^2w(x)dx\le\int_{\real^d}|f(y)|w(y)dy.
  \end{align*}
This is the desired estimate for the short variation. Thus we have proved the weak type $(1,1)$ inequality of Theorem~\ref{vq D}. \cqd


\section{Proof of Theorem~\ref{vq vector}}


This section is devoted to the proof of Theorem~\ref{vq vector}. We will only consider the case of singular integrals, that of the differential operators being handled in a similar way.

First note that Theorem~\ref{vq vector} clearly holds for $p=\rho$ thanks to Theorem~\ref{vq CZ}. Then we deduce the type  $(p, p)$ inequality for any $1<p<\8$ by extrapolation techniques  described in sections~IV.5 and V.6 of \cite{gar-rubio} (see, in particular, Remark~V.6.5 there).

We are thus left to prove the weak type $(1, 1)$ inequality. This proof is similar to but more complicated than that of the same inequality in the scalar case presented in section~\ref{pf of weakA}.  We will follow the structure of that proof. The steps~1-5 mentioned in the sequel are those in section~\ref{pf of weakA}.

Let $f :\real^d\to\el^\rho$ be a compactly supported integrable function, so $f=\{f_n\}_n$. Let $\f$ be the function given by $\f(x)=\|f(x)\|_\rho$ (the norm here is, of course, that of $\el^\rho$). We now apply Lemma~\ref{CZD} to $\l=2$ and $\f$, and keep all notation introduced in  section~\ref{pf of weakA}. In particular, $f$ is written as the sum of its good and bad parts: $f=g+b$. Both $g$ and $b$ take values in $\el^\rho$. We set $g=\{g_n\}_n$,  $b=\{b_n\}_n$ and $b_i=\{b_{n,i}\}_n$. Note that each $b_{n,i}$ is supported on the cube $Q_i$ and its mean vanishes.

We must show
  \begin{align*}
  w\big(\big\{x\in\real^d: \sum_n\big(\V_q\K(g_n)(x)\big)^\rho>1\big\}\big)
  &\lesssim\int_{\real^d}\|f(x)\|_\rho w(x)dx,\\
  w\big(\big\{x\in\real^d: \sum_n\big(\V_q\K(b_n)(x)\big)^\rho>1\big\}\big)
  &\lesssim\int_{\real^d}\|f(x)\|_\rho w(x)dx.
  \end{align*}
The first inequality on $g$ is proved by the type $(\rho, \rho)$ of $\V_q\K$ already observed above. The measure on the left hand side of the second one is split into two parts, one  on $\wt\O$ and another on $\wt\O^c$. The part on $\wt\O$ is estimated as in section~\ref{pf of weakA}. Thus it remains to show
 \beq\label{b out}
 w\big(\big\{x\in\wt\O^c: \sum_n\big(\V_q\K(b_n)(x)\big)^\rho>1\big\}\big)
  \lesssim\int_{\real^d}\|f(x)\|_\rho w(x)dx.
  \eeq
To this end we will follow steps~1-5  and indicate only the necessary modifications.

\medskip

As in step~1,  for every $x\notin \wt\O$ and $n$ choose an increasing sequence $\{t_{n,j}\}$ such that
 $$\V_{q}\K(b_n) (x)\le 2\big(\sum_j\big|K_{(t_{n,j},\,t_{n,j+1}]}(b_n)(x)\big|^q\big)^{1/q}.$$
Let $I_{n,j}=(t_{n,j},\,t_{n,j+1}]$. Choose $r$ such that $1<r\le\min(q,\,\rho)$. Then
 \begin{align*}
 \Big(\sum_n\big[\sum_j\big|K_{I_{n,j}}(b_n)(x)\big|^q\big]^{\rho/q}\Big)^{1/\rho}
 &\le \Big(\sum_n\big[\sum_{j}\big|\sum_{i\in\I_1(I_{n,j})}K_{I_{n,j}}(b_{n,i})(x)\big|^q\big]^{\rho/q}\Big)^{1/\rho}\\
 &\quad+ \Big(\sum_n\big[\sum_{j}\big|\sum_{i\in\I_2(I_{n,j})} K_{I_{n,j}}(b_{n,i})(x)\big|^q\big]^{\rho/q}\Big)^{1/\rho}\\
 &\le \Big(\sum_n\big[\sum_{j}\big|\sum_{i\in\I_1(I_{n,j})}K_{I_{n,j}}(b_{n,i})(x)\big|\big]^{\rho}\Big)^{1/\rho}\\
 &\quad+ \Big(\sum_n\big[\sum_{j}\big|\sum_{i\in\I_2(I_{n,j})} K_{I_{n,j}}(b_{n,i})(x)\big|^r\big]^{\rho/r}\Big)^{1/\rho}\\
 &\;{\mathop =^{\rm def}}\; A(x)+B(x).
 \end{align*}
Thus \eqref{b out} is reduced to
 \beq\label{interior}
 w\big(\big\{x\in\wt\O^c: A(x)>\frac12\big\}\big)
  \lesssim\int_{\real^d}\|f(x)\|_\rho w(x)dx.
  \eeq
 and
 \beq\label{boundary}
 w\big(\big\{x\in\wt\O^c: B(x)>\frac12\big\}\big)
  \lesssim\int_{\real^d}\|f(x)\|_\rho w(x)dx.
  \eeq
  By step~2, we have
  $$
 \sum_{j}\big|\sum_{i\in\I_1(I_{n,j})}K_{I_{n,j}}(b_{n,i})(x)\big|
 \lesssim \sum_i\frac{l_i^\d}{|x-c_i|^{d+\d}}\int_{Q_i}|f_n(y)|dy.$$
Thus by the Minkowski inequality, 
  \begin{align*}
  A(x)
 \lesssim\sum_i\frac{l_i^\d}{|x-c_i|^{d+\d}}\int_{Q_i}\|f(y)\|_\rho dy
 \end{align*}
As in step~2, we then deduce \eqref{interior}.

The proof of \eqref{boundary} on the boundary sum is more complicated.
As in step~3, for each $n$ we separate the boundary sum into the long and short variations:
 \begin{align*}
 \big(\sum_{j}\big|\sum_{i\in\I_2(I_{n,j})}K_{I_{n,j}}(b_{n,i})(x)\big|^r\big)^{1/r}
 &\le 3^{1/r'} \big(\sum_{I\in\L_n}\big|\sum_{i\in\I_2(I)}K_{I}(b_{n,i})(x)\big|^r\big)^{1/r}\\
  &+ 3^{1/r'}\big( \sum_{I\in\S_n}\big|\sum_{i\in\I_2(I)}K_{I}(b_{n,i})(x)\big|^r \big)^{1/r}.
  \end{align*}
Accordingly,
 $$
 B(x)\le 3^{1/r'}\big( B_1(x)+B_2(x)\big),$$
 where
 \begin{align*}
   B_1(x)
   &=\Big(\sum_n \big[\sum_{I\in\L_n}\big|\sum_{i\in\I_2(I)}K_{I}(b_{n,i})(x)\big|^r\big]^{\rho/r}\Big)^{1\rho}\\
  B_2(x)
  &=\Big(\sum_n \big[\sum_{I\in\S_n}\big|\sum_{i\in\I_2(I)}K_{I}(b_{n,i})(x)\big|^r\big]^{\rho/r}\Big)^{1\rho}.
  \end{align*}
 Thus \eqref{boundary}  is split into two inequalities
  \beq\label{boundary1}
  w\big(\big\{x\in\wt\O^c: B_1(x)>\frac1{4\cdot 3^{1/r'}}\big\}\big)
  \lesssim\int_{\real^d}\|f(x)\|_\rho w(x)dx.
  \eeq
 and
 \beq\label{boundary2}
 w\big(\big\{x\in\wt\O^c: B_2(x)>\frac1{4\cdot 3^{1/r'}}\big\}\big)
  \lesssim\int_{\real^d}\|f(x)\|_\rho w(x)dx.
  \eeq

 We now apply \eqref{long} to each $\L_n$ (with $r$ instead of $2$). Let $\I^n_{2,k}$ be the subset associated to $\L_n$ defined just before \eqref{long}. Then we have
 \begin{align*}
  \sum_{I\in\L_n}\big|\sum_{i\in\I_2(I)}K_{I}(b_{n,i})(x)\big|^r
  &\le 2^{r-1}\sum_{k\in\ent}\big|\sum_{i\in\I^n_{2,k}}K_{D_k}(b_{n,i})(x)\big|^r\\
   &\le 2^{r-1}\sum_{k\in\ent}\big(\sum_j\sum_{i\in\I^n_{2,k}\cap\I_j}\big|K_{D_k}(b_{n,i})(x)\big|\big)^r.
   \end{align*}
Therefore, by the Minkowski inequality (recalling that $r\le\rho$),
 \begin{align*}
   B_1(x)
   &\lesssim\Big(\sum_{k\in\ent}\big(\sum_j\big[\sum_n\sum_{i\in\I^n_{2,k}\cap\I_j}
   \big|K_{D_k}(b_{n,i})(x)\big|^{\rho}\big]^{1/\rho}\big)^r\Big)^{1/r}\\
   &=\Big(\sum_{k\in\ent}\big(\sum_jh_{k,j}(x)\big)^r\Big)^{1/r},
   \end{align*}
where
 $$h_{k,j}(x)=\big[\sum_n\sum_{i\in\I^n_{2,k}\cap\I_j}  \big|K_{D_k}(b_{n,i})(x)\big|^{\rho}\big]^{1/\rho}.$$
Hence
 \begin{align*}
 w\big(\big\{x\in\wt\O^c: B_1(x)>\frac1{4\cdot 3^{1/r'}}\big\}\big)
  \lesssim   \int_{\wt\O^c}\sum_{k\in\ent}\big(\sum_jh_{k,j}(x)\big)^rw(x)dx.
 \end{align*}
However,
  \begin{align*}
   h_{k,j}(x)
  &\lesssim \frac1{2^{dk}}\big[\sum_n\big(\sum_{i\in\I_{2,k}^n\cap\I_j}\int_{Q_i}|b_{n,i}(y)|dy\big)^\rho\big]^{1/\rho}\\
  &\le\frac1{2^{dk}}\big[\sum_n\big(\sum_{i\in\I_{2}(D_k)\cap\I_j}\int_{Q_i}|b_{n,i}(y)|dy\big)^\rho\big]^{1/\rho}\\
  &\le\frac1{2^{dk}}\sum_{i\in\I_{2}(D_k)\cap\I_j}\int_{Q_i}\|b_{i}(y)\|_\rho dy.
  \end{align*}
We can now apply the argument in step~4 to get the following inequality
  $$\int_{\wt\O^c}h_{k,j}(x)^rw(x)dx\lesssim 2^{(r-1)(j-k)}\Delta_j,\quad k>j,$$
 where
  $$\Delta_j=\sum_{i\in\I_j}\int_{Q_i}\|f(y)\|_\rho w(y)dy.$$
Therefore,  by Lemma~\ref{AOL}, we deduce
 $$\int_{\wt\O^c}\sum_{k\in\ent}\big(\sum_jh_{k,j}(x)\big)^rw(x)dx
 \lesssim \sum_j\Delta_j\le\int_{\real^d}\|f(y)\|_\rho w(y)dy.$$
Combining the preceding inequalities, we get \eqref{boundary1}.

To complete the proof of  Theorem~\ref{vq vector}, it only remains to show \eqref{boundary2}. It is this part that had prevented us from completing the paper since its first draft was written in the fall of 2012. We are very grateful to Guixiang Hong who kindly communicated to us the following argument, inspired by \cite{kzk}. 

By the weak type $(1, 1)$ of the Hardy-Littlewood maximal function  with respect to $A_1$ weights,  \eqref{boundary2} will follow from the following inequality
 \beq\label{boundary2bis}
  \int_{\wt\O^c}B_2(x)^rw(x)dx\lesssim\sum_i w(Q_i).
  \eeq
Let $\S_{n,k}=\{I\in\S_n: I\subset(2^k,\, 2^{k+1}]\}$. As already observed before, for $I\in\S_{n,k}$ and $i\in\I_2(I)\cap\I_j$,  we have that $k>j$ and  $2^k<d(x, \,Q_i)<2^{k+1}$. Now choose $\a$ such that $(d-1)/r'<\a<d/r'$. Then for any $I\in\S_{n,k}$, by the H\"older inequality,
  $$
  \big|\sum_{i\in\I_2(I)\cap\I_j}K_{I}(b_{n,i})(x)\big|^r
  \le \big(\sum_{j<k}\sum_{i\in\I_2(I)\cap\I_j}2^{(k-j)\a r} |K_{I}(b_{n,i})(x)|^r\big)\,\big(\sum_{j<k}\sum_{i\in\I_2(I)\cap\I_j}2^{(j-k)\a r'} \big)^{\frac{r}{r'}}
  $$
For every $i\in\I_2(I)\cap\I_j$, $Q_i$ is a dyadic cube of side length $2^j$ and intersects the boundary of the annulus $x+R_I$; so the number of a disjoint collection of such cubes is controlled by $2^{(d-1)(k-j)}$. Thus by the choice of $\a$, the second double series on the above right-hand side  is convergent. Therefore, letting $U_k=\{i\,:\, l_i<2^k,\, d(x, \,Q_i)<2^{k+1}\}$, we then deduce
 \begin{align*}
 \sum_{I\in\S_{n,k}}\big|\sum_{i\in\I_2(I)\cap\I_j}K_{I}(b_{n,i})(x)\big|^r
 &\lesssim \sum_{I\in\S_{n,k}}\sum_{j<k}\sum_{i\in\I_2(I)\cap\I_j}2^{(k-j)\a r} |K_{I}(b_{n,i})(x)|^r\\
 &\lesssim \sum_{i\in U_k} l_i^{-\a r}\,2^{k\a r} \sum_{I\in\S_{n,k}} |K_{I}(b_{n,i})(x)|^r\\
 &\lesssim \sum_{i\in U_k} l_i^{-\a r}\,2^{k(\a-d) r} \Big(\int_{Q_i}|b_{n,i}(y)|dy\Big)^r\,.
  \end{align*}
Then by the Minkowski inequality and Lemma~\ref{CZD}, we have
 \begin{align*}
 B_2(x)^r
 &\lesssim \sum_k \sum_{i\in U_k} l_i^{-\a r}\,2^{k(\a-d) r} \Big(\int_{Q_i}\|b_{i}(y)\|_{\rho}dy\Big)^r\\
 &\lesssim \sum_k \sum_{i\in U_k} l_i^{(d-\a) r}\,2^{k(\a-d) r}\,.
  \end{align*}
Using the properties of $A_1$ weights and by the choice of $\a$, we finally get
 \begin{align*}
 \int_{\wt\O^c}B_2(x)^rw(x)dx
 &\lesssim \sum_{i} l_i^{(d-\a) r}\sum_{k: 2^k>l_i} 2^{k(\a-d) r}  \int_{\wt\O^c\cap\{x:d(x,\, Q_i)<2^{k+1}\}} w(x)dx\\
 &\lesssim \sum_{i} l_i^{(d-\a) r}\sum_{k: 2^k>l_i} 2^{k(\a-d) r} \,2^{kd} \inf_{x\in Q_i} w(x)\\
 &\lesssim \sum_{i} l_i^{(d-\a) r}\, l_i^{(\a-d) r+d}  \inf_{x\in Q_i} w(x)\\
 &\lesssim \sum_{i} l_i^{d}  \inf_{x\in Q_i} w(x)\le \sum_{i} w(Q_i).
  \end{align*}
This is the desired inequality \eqref{boundary2bis}. So Theorem~\ref{vq vector} is completely proved. \cqd

\medskip

\noindent{\bf Acknowledgments.} We are very grateful to Guixiang Hong for helping us to fix up the proof of the weak type $(1, 1)$ inequality in Theorem~\ref{vq vector}. We also acknowledge the financial supports of  ANR-2011-BS01-008-01, NSFC grant (No. 11271292, 11301401 and 11431011) and MTM2011-28149-C02-01.

\bigskip



\begin{thebibliography}{10}

\bibitem{Bour}
 J. Bourgain.
 \newblock  Pointwise ergodic theorems for arithmetic sets.
 \newblock  \textit{Publ. Math. IHES.} 69 (1989), 5-41.

\bibitem{CJRW1}
 J. T. Campbell, R. L. Jones, K. Reinhold, and M. Wierdl.
 \newblock  Oscillation and variation for the Hilbert transform.
\newblock  \textit{Duke Math. J.} 105 (2000), 59-83.

\bibitem{CJRW2}
 J. T. Campbell, R. L. Jones, K. Reinhold, and M. Wierdl.
 \newblock  Oscillation and variation for singular integrals in higher dimensions.
\newblock  \textit{ Trans. Amer. Math. Soc.} 355 (2003), 2115-2137.

\bibitem{CMMTV}   R. Crescimbeni, R. A. Mac\' ias, T. Men\' arguez, J. L. Torrea, and
B. Viviani.
 \newblock The $\rho$-variation as an operator between maximal operators and singular integrals.
 \newblock  \textit{J. Evol. Equ.} 9 (2009), 81-102.

 \bibitem{DMC}
  Y. Do, C. Muscalu, and C. Thiele
 \newblock Variational estimates for paraproducts.
 \newblock  \textit{Rev. Mat. Iberoam.} 28 (2012),  8578-78.

 \bibitem{FS}
 C. Fefferman, and E. M. Stein.
\newblock Some maximal inequalities.
  \newblock \textit{  Amer. J. Math.}  93 (1971), 107-115.

\bibitem{gar-rubio}
   J.~Garc{\'{\i}}a-Cuerva, and J.L. Rubio de Francia.
 \newblock \textit{Weighted norm inequalities and related topics}.
 \newblock   North-Holland Publishing Co., Amsterdam, 1985.

\bibitem{GT}
 A.T. Gillespie, and J.L. Torrea.
 \newblock Dimension free estimates for the oscillation of Riesz transforms.
 \newblock \textit{ Isreal J. Math.}  141 (2004), 125-144.

\bibitem{HMS}
 E. Harboure, Mac{\'{\i}}as, and  C. Segovia.
\newblock Extrapolation results for classes of weights.
\newblock \textit{Trans. Amer. Math. Soc.} 110 (1988),  383-397.

\bibitem{HM}
 G. Hong, and T. Ma. 
\newblock Vector valued $q$-variation for ergodic averages and analytic semigroups.
\newblock To appear (arXiv:1411.1254).

\bibitem{HLP}
 T.P. Hyt\"onen, M.T. Lacey, and C. P\'erez.
\newblock Non-probabilistic proof of the $A_2$ theorems and sharp weighted bounds for the $q$-variation of singular integrals.
\newblock  \textit{Bull. Lond. Math. Soc.} 45 (2013), 529-540.

\bibitem{jrw-high}
 R.L. Jones,  J.M. Rosenblatt, and M. Wierdl.
\newblock Oscillation in ergodic theory: higher dimensional results.
\newblock \textit{ Isreal J. Math. } 135 (2003), 1-27.

\bibitem{jsw}
 R. L. Jones, A. Seeger, and J. Wright.
 \newblock Strong variational and jump inequalities in harmonic analysis.
 \newblock \textit{ Trans. Amer. Math. Soc.} 360 (2008), 6711-6742.


\bibitem{kaufman}
 R.L. Jones, R. Kaufman, J.M. Rosenblatt, and M. Wierdl.
 \newblock Oscillation in ergodic theory.
 \newblock   \textit{Ergod. Th.  Dynam. Sys. }  18 (1998),  889-935.
 
\bibitem{kzk} 
 B. Krause, and P. Zorin-Kranich.
  \newblock Weighted and vector-valued variational estimates for ergodic averages.
 \newblock  Preprint, 2015 (arXiv :1 409.7120).

\bibitem{LMX}
 C. Le Merdy, and Q. Xu.
 \newblock Strong $q$-variation inequalities for analytic semigroups. 
\newblock   \textit{Ann. Inst. Fourier}, 62 (2012), 2069-–2097.

\bibitem{Le}
 D. L\'epingle.
 \newblock La variation d'ordre $p$ des semi-martingales.
\newblock   \textit{Z. Wahr.Verw. Gebiete}  36  (1976), 295-316.

\bibitem{MTX}
  T. Ma, J.L. Torrea, and Q. Xu.
 \newblock Weighted variation inequalities for differential operators and singular integrals.
 \newblock  \textit{J. Funct. Anal.}  265 (2013), 545-561.

 \bibitem{MTo}
 A. Mas, and X. Tolsa.
  \newblock  Variation and oscillation for singular integrals with odd kernel on Lipschitz graphs.
 \newblock  \textit{Proc. London Math. Soc.} 105 (2012),  49-86.

 \bibitem{OSTTW}
  R. Oberlin, A. Seeger, T. Tao, C. Thiele, and J. Wright.
   \newblock   A variation norm Carleson theorem.
  \newblock \textit{ J. Eur. Math. Soc.} 14 (2012), 421-464.

 \bibitem{PX}
 G. Pisier, and Q. Xu.
 \newblock  The strong $p$-variation of martingale and orthogonal series.
 \newblock \textit{Probab. Th. Rel. Fields} 77 (1988), 497-514.

\end{thebibliography}
\end{document}